\documentclass[10pt]{amsart}

\usepackage{amsmath,amsthm,amssymb,array,color,fullpage,graphicx,appendix,multirow,subfig,longtable,bbold,mathabx,latexsym,setspace,pinlabel,url}
\usepackage[top=1in,bottom=1in,left=1in,right=1in]{geometry} 
\input xypic

\def\A{\mathcal{A}} 
\def\B{\mathcal{B}}

\newcommand{\aug}{\operatorname{Aug}}

\def\tepsilon{\tilde\epsilon}
\def\ta{\tilde{a}}
\def\tb{\tilde{b}}
\def\tc{\tilde{c}}
\def\tpartial{\tilde\partial}
\def\tv{\tilde{v}}

\newtheorem*{thmLinksInR}{Theorem \ref{thm:linksInR}}
\newtheorem*{thmLinksInSolidTorus}{Theorem \ref{thm:correspSolidTorus}}
\newtheorem*{corNonvanishing}{Corollary \ref{cor:nonvanishing}}
\newtheorem*{thmRulingPoly}{Theorem \ref{thm:rulingPoly}}

\def\clha{LHA(\Lambda)}
\def\clho{LHO(\Lambda)}

\def\ccheck{\widecheck{LHO^+}(\Lambda)}
\def\chat{\widehat{LHO^+}(\Lambda)}
\def\clhho{LH^{Ho}(\Lambda)}
\def\hlhho{L\mathbb{H}^{Ho}(\Lambda)}
\def\hsh{S\mathbb{H}(X)}

\def\integers{{\mathbb{Z}}}
\def\reals{{\mathbb R}} 

\def\rationals{{\mathbb Q}}

\def\ker{{\hbox{ker}}\,}
\def\other{\text{otherwise}}

\def\bigno{\bigskip\noindent}

\newcommand{\diag}[1]{\begin{center}\begin{minipage}{5in}\xymatrix{#1}\end{minipage}\end{center}}
\newcommand{\seq}[1]{\begin{center}\begin{minipage}{5in}\xymatrix@R=.125pc{#1}\end{minipage}\end{center}}
 
\def\arr{\ar[r]}

\def\mapstol{\ar@{|->}[l]} 
\def\mapstor{\ar@{|->}[r]}
\def\mapstod{\ar@{|->}[d]}
\def\mapstou{\ar@{|->}[u]}

\newtheorem{lem}{Lemma}[section]
\newtheorem{thm}[lem]{Theorem}

\newtheorem{cor}[lem]{Corollary}
\newtheorem{prop}[lem]{Proposition}

\theoremstyle{definition} 
\newtheorem{defn}[lem]{Definition}
\newtheorem{ex}[lem]{Example} 

\newtheorem{rmk}[lem]{Remark} 
\newtheorem{note}[lem]{Notation}

\onehalfspace

\newcolumntype{C}[1]{>{\centering\arraybackslash$}p{#1}<{$}}

\begin{document}
\title{Augmentations and Rulings of Legendrian Links in $\#^k(S^1\times S^2)$}
\author{C. Leverson}
\address{Duke University, Durham, NC 27708}
\email{cleverso@math.duke.edu}
\date{\today}
\maketitle

\begin{abstract}
  Given a Legendrian link in $\#^k(S^1\times S^2)$, we extend the
  definition of a normal ruling from $J^1(S^1)$ given by Lavrov and
  Rutherford and show that the existence of an augmentation to any
  field of the Chekanov-Eliashberg differential graded algebra over
  $\mathbb{Z}[t,t^{-1}]$ is equivalent to the existence of a normal
  ruling of the front diagram. For Legendrian knots, we also show that
  any even graded augmentation must send $t$ to $-1$. We use the
  correspondence to give nonvanishing results for the symplectic
  homology of certain Weinstein $4$-manifolds. We show a similar
  correspondence for the related case of Legendrian links in
  $J^1(S^1)$, the solid torus.
\end{abstract}

\section{Introduction}

Augmentations and normal rulings are important tools in the study of
Legendrian knot theory, especially in the study of Legendrian knots in
$\reals^3$. Here, augmentations are augmentations of the
Chekanov-Eliashberg differential graded algebra introduced by Chekanov
in \cite{Chekanov} and Eliashberg in
\cite{EliashbergInvariants}. Chekanov describes the noncommutative
differential graded algebra (DGA) over $\integers/2$ associated to a
Lagrangian diagram of a Legendrian link in
$(\reals^3,\xi_{\text{std}})$ combinatorially: The DGA is generated by
crossings of the link; the differential is determined by a count of
immersed polygons whose corners lie at crossings of the link and whose
edges lie on the link. This is called the Chekanov-Eliashberg DGA and
Chekanov showed that the homology of this DGA is invariant under
Legendrian isotopy. Etnyre, Ng, and Sabloff defined a lift of the
Chekanov-Eliashberg DGA to a DGA over $\integers[t,t^{-1}]$ in
\cite{EtnyreInvariants}. Following ideas introduced by Eliashberg in
\cite{EliashbergWave}, Fuchs \cite{FuchsAug} and Chekanov-Pushkar
\cite{ChekanovFronts} gave invariants of Legendrian knots in
$\reals^3$ using generating families, functions whose critical values
generate front diagrams of Legendrian knots, by decomposing the
generating families. These are generally called ``normal rulings.''

These two invariants are very closely related; Fuchs \cite{FuchsAug},
Fuchs-Ishkhanov \cite{FuchsIshkhanov}, and Sabloff \cite{SabloffAug}
showed that the existence of a normal ruling is equivalent to the
existence of an augmentation to $\integers/2$ of the
Chekanov-Eliashberg DGA $\A$ for Legendrian knots in $\reals^3$. Here,
given a unital ring $S$, an augmentation is a ring map $\epsilon:\A\to
S$ such that $\epsilon\circ\partial=0$ and $\epsilon(1)=1$. One of the
main results of \cite{Leverson} is that the equivalence remains true
when one looks at augmentations to a field of the lift of the
Chekanov-Eliashberg DGA from \cite{EtnyreInvariants} to the DGA over
$\integers[t^{\pm1}]$ for Legendrian knots in $\reals^3$. We extend
the result to Legendrian \emph{links} in $\reals^3$ to prove the main
result of this paper.

\begin{thm}\label{thm:linksInR}
  Let $\Lambda$ be an $s$-component Legendrian link in
  $\reals^3$. Given a field $F$, the Chekanov-Eliashberg DGA
  $(\A,\partial)$ over $\integers[t_1^{\pm1},\ldots,t_s^{\pm1}]$ has a
  $\rho$-graded augmentation $\epsilon:\A\to F$ if and only if a front
  diagram of $\Lambda$ has a $\rho$-graded normal ruling. Furthermore,
  if $\rho$ is even, then $\epsilon(t_1\cdots t_s)=(-1)^s$.
\end{thm}

The final statement tells us that for all even graded augmentations
$\epsilon:\A\to F$, $\epsilon(t_1\cdots t_s)=(-1)^s$. In particular,
if $\Lambda$ is a knot, then any even graded augmentation sends $t$ to
$-1$.

For $k\geq0$, an analogous correspondence can be shown for Legendrian
links in $\#^k(S^1\times S^2)$. A Legendrian link in $\#^k(S^1\times
S^2)$ with the standard contact structure is an embedding
$\Lambda:\coprod_s S^1\to\#^k(S^1\times S^2)$ which is everywhere
tangent to the contact planes. We will think of them as Gompf does in
\cite{GompfHandlebody}. For an example, see
Figure~\ref{fig:exResolution}. In this paper, we extend the definition
of normal ruling of a Legendrian link in $\reals^3$ to a Legendrian
link in $\#^k(S^1\times S^2)$. We can then define the ruling
polynomial for a Legendrian link in $\#^k(S^1\times S^2)$ and show
that the ruling polynomial is invariant under Legendrian isotopy.

\begin{thm}\label{thm:rulingPoly}
  The $\rho$-graded ruling polynomial $R^\rho_{(\Lambda,m)}$ with
  respect to the Maslov potential $m$ (which changes under Legendrian
  isotopy) is a Legendrian isotopy invariant.
\end{thm}

In \cite{Ekholms1s2}, Ekholm and Ng extend the definition of the
Chekanov-Eliashberg DGA over $\integers[t,t^{-1}]$ to Legendrian links
in $\#^k(S^1\times S^2)$. The main result of this paper uses
Theorem~\ref{thm:linksInR} to extend the correspondence between normal
rulings and augmentations to a correspondence for Legendrian links in
$\#^k(S^1\times S^2)$.

\begin{thm}\label{thm:main}
  Let $\Lambda$ be an $s$-component Legendrian link in $\#^k(S^1\times
  S^2)$ for some $k\geq0$. Given a field $F$, the Chekanov-Eliashberg
  DGA $(\A(\Lambda),\partial)$ over
  $\integers[t_1^{\pm1},\ldots,t_s^{\pm1}]$ has a $\rho$-graded
  augmentation $\epsilon:\A(\Lambda)\to F$ if and only if a front
  diagram of $\Lambda$ has a $\rho$-graded normal ruling. Furthermore,
  if $\rho$ is even, then $\epsilon(t_1\cdots t_s)=(-1)^s$.
\end{thm}

\noindent
Notice that one can consider Legendrian links in $\reals^3$ as being
Legendrian links in $\#^0(S^1\times S^2)$. In this way, this result is
a generalization of the correspondence in \cite{Leverson} and
Theorem~\ref{thm:linksInR}.

Along with the work of Bourgeois, Ekholm, and Eliashberg in
\cite{BEE}, Theorem~\ref{thm:main} gives nonvanishing results for
Weinstein (Stein) $4$-manifolds. In particular:

\begin{cor}\label{cor:nonvanishing}
  If $X$ is the Weinstein $4$-manifold that results from attaching
  $2$-handles along a Legendrian link $\Lambda$ to $\#^k(S^1\times
  S^2)$ and $\Lambda$ has a graded normal ruling, then the full symplectic
  homology $\hsh$ is nonzero.
\end{cor}

\noindent
This follows from Theorem~\ref{thm:main} as the existence of a normal
ruling implies the existence of an augmentation to $\rationals$,
which, by \cite{BEE}, is necessary for the full symplectic homology to
be nonzero.

We show a correspondence for Legendrian links in the $1$-jet space of
the circle $J^1(S^1)$. In \cite{NgSolidTorus}, Ng and Traynor extend
the definition of the Chekanov-Eliashberg DGA to Legendrian links in
$J^1(S^1)$.  In \cite{LavrovSolidTorus}, Lavrov and Rutherford extend
the definition of normal ruling to a ``generalized normal ruling'' of
Legendrian links in $J^1(S^1)$ and show that the existence of a
generalized normal ruling is equivalent to the existence of an
augmentation to $\integers/2$ of the Chekanov-Eliashberg DGA of a
Legendrian link in $J^1(S^1)$. In \S\ref{sec:solidTorus}, we show that
this correspondence holds for augmentations to any field of the
Chekanov-Eliashberg DGA over
$\integers[t_1^{\pm1},\ldots,t_s^{\pm1}]$.

\begin{thm}\label{thm:correspSolidTorus}
  Let $\Lambda$ be a Legendrian link in $J^1(S^1)$. Given a field $F$,
  the Chekanov-Eliashberg DGA $(\A,\partial)$ over
  $\integers[t_1^{\pm1},\ldots,t_s^{\pm1}]$ has a $\rho$-graded
  augmentation $\epsilon:\A\to F$ if and only if a front diagram of
  $\Lambda$ has a $\rho$-graded generalized normal ruling.
\end{thm}

\subsection{Outline of the article}
In \S\ref{sec:background} we recall background on Legendrian links in
$\#^k(S^1\times S^2)$ and $\reals^3$. We give definitions of the
Chekanov-Eliashberg DGA over $\integers[t,t^{-1}]$, with sign
conventions, and augmentations of the DGA in both $\#^k(S^1\times
S^2)$ and $\reals^3$. We also define normal rulings for links in
$\#^k(S^1\times S^2)$ and show that the ruling polynomial is invariant
under Legendrian isotopy. In \S\ref{sec:correspLinks}, we prove
Theorem~\ref{thm:linksInR}. In \S\ref{sec:augToRuling}, given an
augmentation, we construct a normal ruling proving one direction of
Theorem~\ref{thm:main}. In \S\ref{sec:rulingToAug}, given a normal
ruling, we construct an augmentation, finishing the proof of
Theorem~\ref{thm:main}. In \S\ref{sec:solidTorus}, we prove
Theorem~\ref{thm:correspSolidTorus}. In the Appendix, we give the
nonvanishing symplectic homology result.

\subsection{Acknowledgements}
The author thanks Lenhard Ng and Dan Rutherford for many helpful
discussions. This work was partially supported by NSF grants
DMS-0846346 and DMS-1406371.

\section{Background Material}\label{sec:background}
\subsection{Legendrian Links in $\#^k(S^1\times S^2)$}
In this section we will briefly discuss necessary concepts of
Legendrian links in $\#^k(S^1\times S^2)$. We will follow the notation
in \cite{Ekholms1s2}.

\begin{defn}
  Let $A,M>0$. A tangle in $[0,A]\times[-M,M]\times[-M,M]$ is {\bf
    Legendrian} if it is everywhere tangent to the standard contact
  structure $dz-ydx$. Informally, a Legendrian tangle $T$ in
  $[0,A]\times[-M,M]\times[-M,M]$ is in {\bf normal form} if
  \begin{itemize}
  \item $T$ meets $x=0$ and $x=A$ in $k$ groups of strands, where the
    groups are of size $N_1,\ldots,N_k$, from top to bottom in both
    the $xy$ and $xz$ projections,
  \item and within the $\ell$-th group, we label the strands by
    $1,\ldots,N_\ell$ from top to bottom at $x=0$ in both the $xy$ and
    $xz$ projections and $x=A$ in the $xz$ projection, and from bottom
    to top at $x=A$ in the $xy$ projection.
  \end{itemize}
\end{defn}

Every Legendrian tangle in normal form gives a Legendrian link in
$\#^k(S^1\times S^2)$ by attaching $k$ $1$-handles which join parts of
the $xz$ projection of the tangle at $x=0$ to the parts at $x=A$.  In
particular, the $\ell$-th $1$-handle joins the $\ell$-th group at
$x=0$ to the $\ell$-th group at $x=A$ and connects the strands in this
group with the same label at $x=0$ and $x=A$ through the
$1$-handle. See Figure~\ref{fig:exResolution}.

Every Legendrian link in $\#^k(S^1\times S^2)$ has an $xz$-diagram of
the form given by Gompf in \cite{GompfHandlebody}, which we will call
{\bf Gompf standard form}. The left diagram of
Figure~\ref{fig:exResolution} is an example of a link in Gompf
standard form. Any link in Gompf standard form can be isotoped to a
link whose $xy$-projection is obtained from the $xz$-diagram by {\bf
  resolution}. The resolution of an $xz$-diagram of a link is obtained
by the replacements given in Figure~\ref{fig:resolutions}. For an
example, see Figure~\ref{fig:exResolution}. By \cite{Ekholms1s2}, an
$xy$-diagram obtained by the resolution of an $xz$-diagram of a link
in Gompf standard form is in normal form. Thus, we can assume that the
$xy$-diagram of any Legendrian link is in normal form.

\begin{figure}
  \labellist
  \small
  \pinlabel $1$ [r] at 91 96
  \pinlabel $2$ [r] at 91 73
  \pinlabel $3$ [r] at 91 50
  \pinlabel $4$ [r] at 91 27

  \pinlabel $4$ [B] at 538 83
  \pinlabel $3$ [B] at 538 63
  \pinlabel $2$ at 538 43
  \pinlabel $1$ [t] at 538 26

  \pinlabel $b_{34}$ [b] at 407 34
  \pinlabel $b_{24}$ [b] at 438 43
  \pinlabel $b_{14}$ [b] at 471 53
  \pinlabel $b_{23}$ [b] at 462 33
  \pinlabel $b_{13}$ [b] at 490 44
  \pinlabel $b_{12}$ [b] at 512 34

  \endlabellist
  \includegraphics[width=.9\textwidth]{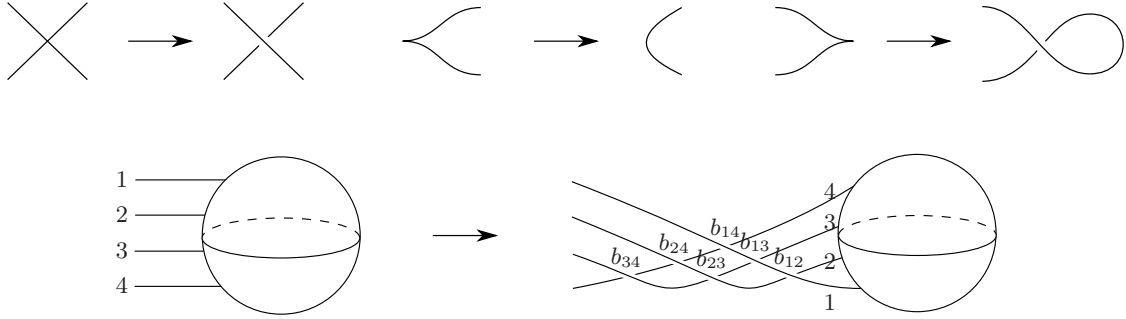}
  \caption{Resolutions of an $xz$-diagram in Gompf standard form.}
  \label{fig:resolutions}
\end{figure}

\begin{figure}
  \labellist
  \small
  \pinlabel $1$ at 165 280
  \pinlabel $2$ at 165 258
  \pinlabel $3$ at 165 235
  \pinlabel $4$ at 165 212

  \pinlabel $1'$ at 165 123
  \pinlabel $2'$ at 165 92

  \pinlabel $1$ at 724 280
  \pinlabel $2$ at 724 258
  \pinlabel $3$ at 724 235
  \pinlabel $4$ at 724 212

  \pinlabel $1'$ at 724 123
  \pinlabel $2'$ at 724 92

  \pinlabel $1$ at 1200 383
  \pinlabel $2$ at 1200 346
  \pinlabel $3$ at 1200 307 
  \pinlabel $4$ at 1200 270

  \pinlabel $1'$ at 1200 117
  \pinlabel $2'$ at 1200 80

  \pinlabel $4$ at 2015 353
  \pinlabel $3$ at 2015 325
  \pinlabel $2$ at 2015 295
  \pinlabel $1$ at 2015 273

  \pinlabel $2'$ [b] at 1853 107
  \pinlabel $1'$ [t] at 1853 84
  \endlabellist
  \includegraphics[width=\textwidth]{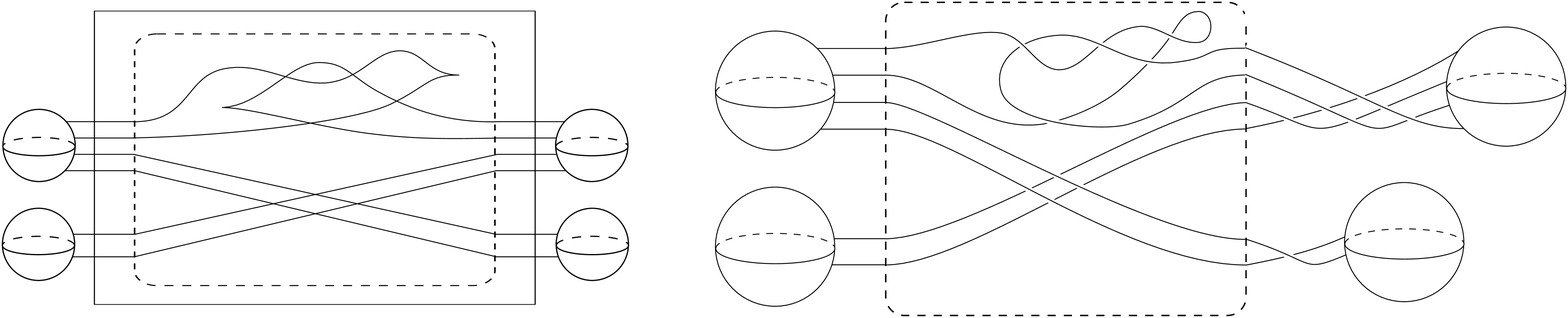}
  \caption{The left gives a Legendrian $xz$-diagram of a link in
    $\#^2(S^1\times S^2)$ in Gompf standard form. The right gives the
    resolution of the Legendrian link to an $xy$-diagram of a
    Legendrian isotopic link.}
  \label{fig:exResolution}
\end{figure}

\subsection{Definition of the DGA and augmentations in $\#^k(S^1\times
  S^2)$} \label{sec:defnDGA}

This section contains an overview of the differential graded algebra
over $\integers[t^{\pm1}_1,\ldots,t^{\pm1}_s]$ presented by Ekholm, Ng
in \cite{Ekholms1s2}. Let $\Lambda=\Lambda_1\coprod\cdots\coprod\Lambda_n$
be a Legendrian link in $\#^k(S^1\times S^2)$, where the $\Lambda_i$
denote the components of $\Lambda$ and $n\leq s$. Let $N_i\geq1$ be
the number of strands of $\Lambda$ which go through the $i$-th
$1$-handle with $N=\sum N_i$ the total number of strands at $x=0$.

\subsection{Internal DGA}
\label{sec:internal}
We will define the internal DGA for a Legendrian link in $S^1\times
S^2$, but one can easily extend the definition to the internal DGA for
a Legendrian link in $\#^k(S^1\times S^2)$ by defining the internal
DGA as follows for each $1$-handle separately.

Let $(r_1,\ldots,r_n)\in\integers^n$ be the $n$-tuple where $r_i$ is
the rotation number of the $i$-th component $\Lambda_i$ and let
$(m(1),\ldots,m(N))\in\integers^N$ be the $N$-tuple of a choice of
Maslov potential for each strand passing through the $1$-handle (see
\S\ref{sec:grading}).

Let $(\A_N,\partial_N)$ denote the DGA defined as follows. Let
$\A$ be the tensor algebra over
$R=\integers[t^{\pm1}_1,\ldots,t^{\pm1}_s]$ generated by $c^0_{ij}$
for $1\leq i<j\leq N$ and $c^p_{ij}$ for $1\leq i,j\leq N$ and
$p\geq1$. Set $\lvert t_i\rvert=-2r_i$, $\lvert t_i^{-1}\rvert=2r_i$,
and
\[\lvert c^p_{ij}\rvert=2p-1+m(i)-m(j)\]
for all $i,j,p$. Define the differential $\partial_N$ on the
generators by
\begin{align*}
  \partial_N(c^0_{ij})&=\sum_{\ell=i+1}^{j-1}(-1)^{\lvert c^0_{i\ell}\rvert+1}c^0_{i\ell}c^0_{\ell j}\\
  \partial_N(c^1_{ij})&=\delta_{ij}+\sum_{\ell=i+1}^N(-1)^{\vert c^0_{i\ell}\rvert+1}c^0_{i\ell}c^1_{\ell j}+\sum_{\ell=1}^{j-1}(-1)^{\lvert c^1_{i\ell}\rvert+1}c^1_{i\ell}c^0_{\ell j}\\
  \partial_N(c^p_{ij})&=\sum_{\ell=0}^p\sum_{m=1}^N(-1)^{\lvert
    c^\ell_{im}\rvert+1}c^\ell_{im}c^{p-\ell}_{mj}
\end{align*}
where $p\geq2$, $\delta_{ij}$ is the Kronecker delta function, and we
set $c^0_{ij}=0$ for $i\geq j$. Extend $\partial_N$ to $\A_N$ by the
Leibniz rule
\[\partial_N(xy)=(\partial_N(x))y+(-1)^{\lvert
  x\rvert}x(\partial_Ny).\] From \cite{Ekholms1s2}, we know
$\partial_N$ has degree $-1$, $\partial_N^2=0$, and
$(\A_N,\partial_N)$ is infinitely generated as an algebra, but is a
filtered DGA, where $c^p_{ij}$ is a generator of the $\ell$-th
component of the filtration if $p\leq\ell$.

Given a Legendrian link $\Lambda\subset\#^k(S^1\times S^2)$, we can
associate a DGA $(\A_{N_i},\partial_{N_i})$ to each of the
$1$-handles. We then call the DGA generated by the collection of
generators of $\A_i$ for $1\leq i\leq k$ with differential induced
by $\partial_{N_i}$, the {\bf internal DGA} of $\Lambda$.

\subsection{Algebra}
Suppose we have a Legendrian link
$\Lambda=\Lambda_1\coprod\cdots\coprod\Lambda_n\subset\#^k(S^1\times S^2)$
in normal form with exactly one point labeled $*_i$ within the tangle
(away from crossings) on each link component $\Lambda_i$ of $\Lambda$
(corresponding to $t_i$). We will discuss the case where there is more
than one base point on a given component in
\S\ref{sec:basePointChanges}.

\begin{note}
  Let $\ta_1,\ldots,\ta_m$ denote the crossings of the tangle diagram
  in normal form. Label the $k$ $1$-handles in the diagram by
  $1,\ldots,k$ from top to bottom. Recall that $N_i$ denotes the
  number of strands of the tangle going through the $i$-th
  $1$-handle. For each $i$, label the strands going through the $i$-th
  $1$-handle on the left side of the diagram $1,\ldots,N_i$ from top
  to bottom and from bottom to top on the right side, as in
  Figure~\ref{fig:exResolution}.
\end{note}

Let $\A(\Lambda)$ be the tensor algebra over
$R=\integers[t^{\pm1}_1,\ldots,t^{\pm1}_s]$ generated by
\begin{itemize}
\item$\ta_1,\ldots,\ta_m$;
\item$c^0_{ij;\ell}$ for $1\leq\ell\leq k$ and $1\leq i<j\leq N_\ell$;
\item$c^p_{ij;\ell}$ for $1\leq\ell\leq k$, $p\geq1$, and $1\leq
  i,j\leq N_\ell$.
\end{itemize}
(In general, we will drop the index $\ell$ when the $1$-handle is
clear.)

\subsection{Grading}
\label{sec:grading}

The following are a few preliminary definitions which will allow us to
define the grading on the generators of $\A(\Lambda)$.

\begin{defn}
  A {\bf path} in $\pi_{xy}(\Lambda)$ is a path that traverses part
  (or all) of $\pi_{xy}(\Lambda)$ which is connected except for where
  it enters a $1$-handle, meaning, where it approaches $x=0$
  (respectively $x=A$) along a labeled strand and exits the $1$-handle
  along the strand with the same label from $x=A$ (respectively
  $x=0$). Note that the tangent vector in $\reals^2$ to the path
  varies continuously as we traverse a path as the strands entering
  and exiting $1$-handles are horizontal.

  The {\bf rotation number} $r(\gamma)$ of a path $\gamma$ is the
  number of counterclockwise revolutions around $S^1$ made by the
  tangent vector $\gamma'(t)/\lvert\gamma'(t)\rvert$ to $\gamma$ as we
  transverse $\gamma$. Generally this will be a real number, but will
  be an integer if and only if $\gamma$ is smooth and closed.
\end{defn}

Thus, the rotation number $r_i=r(\Lambda_i)$ is the rotation number of
the path in $\pi_{xy}(\Lambda)$ which begins at the base point $*_i$
on the link component $\Lambda_i$ and traverses the link component,
following the orientation of the component. In the case where
$\Lambda$ is a link with components $\Lambda_1,\ldots,\Lambda_n$, we
define
\[r(\Lambda)=\gcd(r_1,\ldots,r_n).\] Define
\[\lvert t_i\rvert=-2r(\Lambda_i).\]

If $\pi_{xy}(\Lambda)$ is the resolution of an $xz$-diagram of an
$n$-component link in Gompf standard form, then the method assigning
gradings follows: Choose a {\bf Maslov potential} $m$ that associates
an integer modulo $2r(\Lambda)$ to each strand in the tangle $T$
associated to $\Lambda$, minus cusps and base points, such that the
following conditions hold:
\begin{enumerate}
\item for all $1\leq\ell\leq k$ and all $1\leq i\leq N_\ell$, the
  strand labeled $i$ going through the $\ell$-th $1$-handle at $x=0$
  and the $x=A$ must have the same Maslov potential;
\item if a strand is oriented to the right, meaning it enters the
  $1$-handle at $x=A$ and exits at $x=0$, then the Maslov potential of
  the strand must be even. Otherwise the Maslov potential of the
  strand must be odd;
\item at a cusp, the upper strand (strand with higher $z$-coordinate)
  has Maslov potential one more than the lower strand.
\end{enumerate}
The Maslov potential is well-defined up to an overall shift by an even
integer for knots. (In \cite{Ekholms1s2}, Ekholm and Ng give another
method for defining the gradings using the rotation numbers of
specified paths.)

Set $\lvert t_i\rvert=-2r(\Lambda_i)$ and $\lvert
c^p_{ij;\ell}\rvert=2p-1+m(i)-m(j)$, where $m(i)$ means the Maslov
potential of the strand with label $i$ going through the $\ell$-th
$1$-handle. It remains to define the grading on crossings in the
tangle, crossings resulting from resolving right cusps, and crossings
from the half-twists in the resolution. If $a$ is crossing of tangle
$T$, then let \[\lvert a\rvert=m(S_o)-m(S_u),\] where $S_o$ is the
strand which crosses over the strand $S_u$ at $a$ in the
$xy$-projection of $\Lambda$. If $a$ is a right cusp, define $\lvert
a\rvert=1$ (assuming there is not a base point in the loop). If $a$ is
a crossing in one of the half-twists in the resolution where strand
$i$ crosses over strand $j$ ($i<j$), then
\[\lvert a\rvert=m(i)-m(j).\]

\subsection{Differential}
It suffices to define the differential $\partial$ on generators and
extend by the Leibniz rule. Define
$\partial(\integers[t^{\pm1}_1,\ldots,t^{\pm1}_s])=0$. Set
$\partial=\partial_{N_{\ell}}$ on $\A_{N_{\ell}}$ as in
\S\ref{sec:internal}.

In \cite{Ekholms1s2}, the DGA on crossings $a_i$ is defined by looking
for immersed disks in the $xy$-diagrams of Legendrian links, (see the
left diagram in Figure \ref{fig:linkDip}). However, Ekholm and Ng note
that it is equivalent to look for immersed disks in dip versions of
the diagram, (see the right diagram in Figure~\ref{fig:linkDip}). See
Figure~\ref{fig:dipLabels} for the labeling of the crossings in
Figure~\ref{fig:linkDip}.

\begin{figure}
  \labellist
  \small
  \pinlabel $1$ [B] at 207 382
  \pinlabel $2$ [B] at 207 345
  \pinlabel $3$ [B] at 207 307
  \pinlabel $4$ [B] at 207 270

  \pinlabel $1'$ [B] at 207 117
  \pinlabel $2'$ [B] at 207 81

  \pinlabel $4$ [B] at 1016 361
  \pinlabel $3$ [B] at 1016 331
  \pinlabel $2$ [B] at 1016 299
  \pinlabel $1$ [B] at 1016 273

  \pinlabel $2'$ [B] at 862 109
  \pinlabel $1'$ [B] at 862 85

  \pinlabel $1$ [B] at 1454 382
  \pinlabel $2$ [B] at 1454 345
  \pinlabel $3$ [B] at 1454 307
  \pinlabel $4$ [B] at 1454 270

  \pinlabel $1'$ [B] at 1454 117
  \pinlabel $2'$ [B] at 1454 80

  \pinlabel $1$ [B] at 2000 382
  \pinlabel $2$ [B] at 2000 346
  \pinlabel $3$ [B] at 2000 307
  \pinlabel $4$ [B] at 2000 271

  \pinlabel $1'$ [B] at 2000 118
  \pinlabel $2'$ [B] at 2000 79
  \endlabellist
  \includegraphics[width=\textwidth]{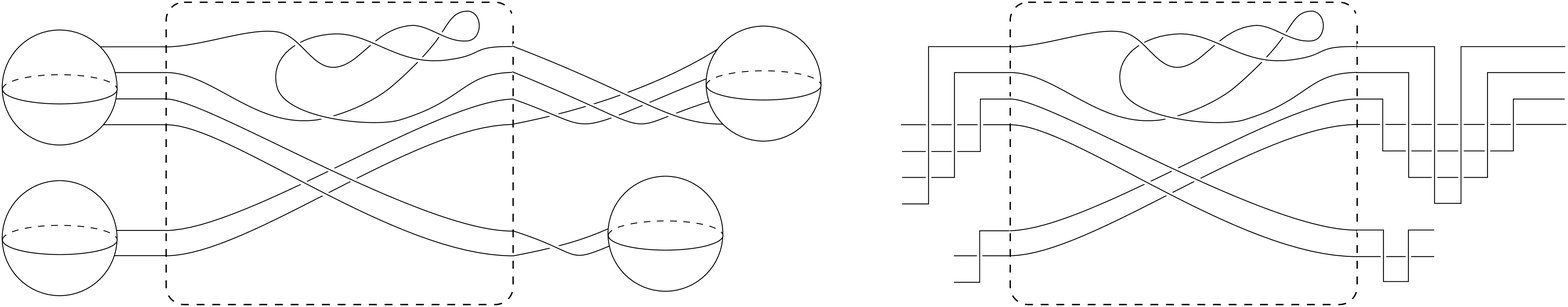}
  \caption{The left gives a Legendrian $xy$-diagram of a link in
    $\#^2(S^1\times S^2)$ which has resulted from the resolution of a
    link in Gompf standard form. The right gives the dipped version of
    the link where the half of a dip on the left side of the dipped
    version is identified with the right half of the dip on the right
    side. See Figure~\ref{fig:dipLabels} for the labeling of the
    crossings in the dips.}
  \label{fig:linkDip}
\end{figure}

\begin{figure}
  \labellist
  \pinlabel $4$ [r] at 0 99
  \pinlabel $3$ [r] at 0 132
  \pinlabel $2$ [r] at 0 164
  \pinlabel $1$ [r] at 0 198

  \pinlabel $b_{14}$ [tr] at 100 99
  \pinlabel $b_{13}$ [tr] at 100 67
  \pinlabel $b_{12}$ [tr] at 100 35
  \pinlabel $b_{24}$ [tr] at 68 99
  \pinlabel $b_{23}$ [tr] at 68 67
  \pinlabel $b_{34}$ [tr] at 37 99

  \pinlabel $c^0_{14}$ [tl] at 147 99
  \pinlabel $c^0_{13}$ [tl] at 147 67
  \pinlabel $c^0_{12}$ [tl] at 147 35
  \pinlabel $c^0_{24}$ [tl] at 179 99
  \pinlabel $c^0_{23}$ [tl] at 179 67
  \pinlabel $c^0_{34}$ [tl] at 212 99
  \endlabellist
  \includegraphics[width=.35\textwidth]{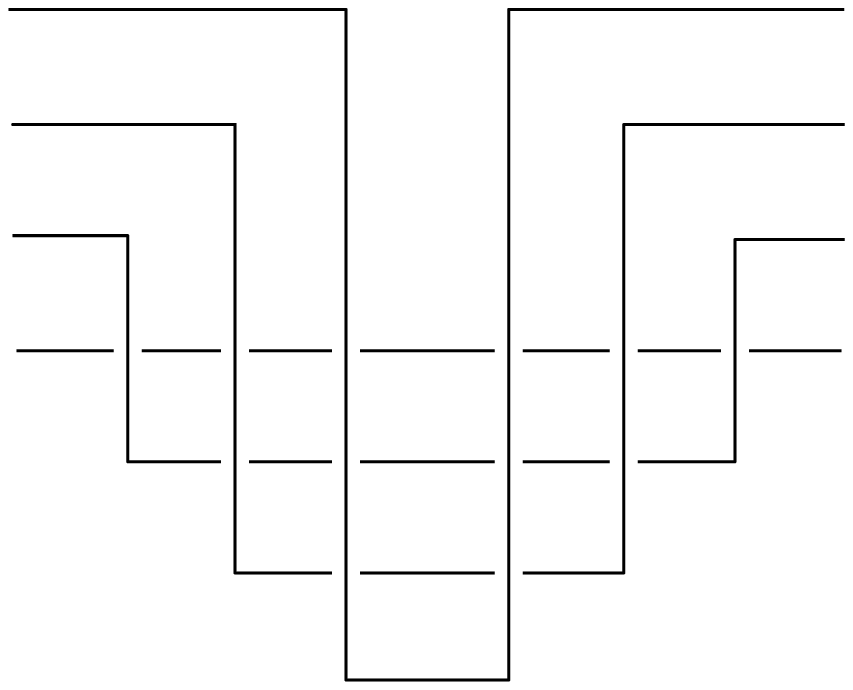}
  \caption{This is the dip at the right of the right figure in
    Figure~\ref{fig:linkDip} with strands and crossings labeled. The
    labels of the partial dip at the left of the right figure in
    Figure~\ref{fig:linkDip} are the same as the right half of the dip
    depicted.}
  \label{fig:dipLabels}
\end{figure}

\begin{defn}
  Let $a,b_1,\ldots,b_\ell$ be generators. Define
  $\Delta(a;b_1,\ldots,b_\ell)$ to be the set of
  orientation-preserving immersions
  \[f:D^2\to\reals^2\] (up to smooth reparametrization) that map
  $\partial D^2$ to the dip version of $\Lambda$ such that
  \begin{enumerate}
  \item $f$ is a smooth immersion except at $a,b_1,\ldots,b_\ell$,
  \item $a,b_1,\ldots,b_\ell$ are encountered as one traverses
    $f(\partial D^2)$ counterclockwise,
  \item near $a,b_1,\ldots,b_\ell$, $f(D^2)$ covers exactly one
    quadrant, specifically, a quadrant with positive Reed sign near
    $a$ and a quadrant with negative Reeb sign near
    $b_1,\ldots,b_\ell$, where the Reeb sign of a quadrant near a
    crossing is defined as in Figure~\ref{fig:orientation}.
  \end{enumerate}
\end{defn}

To each immersed disk, we can assign a word in $\A(\Lambda)$ by
starting with the first corner where the quadrant covered has negative
Reeb sign, $b_1$, and listing the crossing labels of all negative
corners as encountered while following the boundary of the immersed
polygon counterclockwise, $b_1\cdots b_\ell$. We associate an {\bf
  orientation sign} $\delta_{Q,a}$ to each quadrant $Q$ in the
neighborhood of a crossing $a$, defined in
Figure~\ref{fig:orientation}, and use these to define the sign of a
disk $f(D^2)$ to be the product of the orientation signs over all the
corners of the disk. We denote this sign by $\delta(f)$. In many cases
there is a unique disk with positive corner at $a$ (with respect to
Reeb sign) and negative corners at $b_1,\ldots,b_\ell$ and in these we
define $\delta(a;b_1,\ldots,b_\ell)$ to be the sign of the unique
disk. (In exceptional cases there may be more than one disk with
positive corner at $a$ and negative corners at $b_1,\ldots,b_\ell$.)

\begin{figure}
  \labellist
  \pinlabel $-$ [b] at 56 57
  \pinlabel $-$ [t] at 56 57
  \pinlabel $+$ [l] at 56 57
  \pinlabel $+$ [r] at 56 57

  \pinlabel $+$ [br] at 248 55
  \pinlabel $+$ [tl] at 248 55
  \pinlabel $-$ [bl] at 248 55
  \pinlabel $-$ [tr] at 248 55
  \endlabellist

  \includegraphics[width=2.5in]{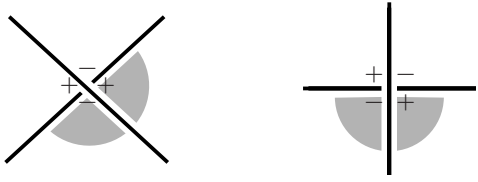}
  \caption{The signs in the figure give the Reeb signs of the
    quadrants around the crossings. The orientation signs are $+1$ for
    all quadrants of crossings of odd degree. For crossings of even
    degree, we use the convention indicated in the left figure if the
    crossing comes from the $xz$-projection and the convention in the
    right figure if the crossing is in a dip, which will be discussed
    in \S \ref{sec:dips}, where the shaded quadrants have orientation
    sign $-1$ and the other quadrants have orientation sign $+1$.}
  \label{fig:orientation}
\end{figure}

Define $n_{*_i}(f)$ or $n_{*_i}(a;b_1,\ldots,b_\ell)$ to be the signed
count of the number of times one encounters the base point $*_i$ while
following $f(\partial D^2)$ counterclockwise, where the sign is
positive if we encounter the base point while following the
orientation of the link component and negative if we encounter the
base point while going against the orientation.

We define
\[\partial(a_i)=\sum_{\ell\geq0}\sum_{(b_1,\ldots,b_\ell)}\sum_{f\in\Delta(a_i;b_1,\ldots,b_\ell)}\delta(f)t_1^{n_{*_1}(f)}\cdots
t_s^{n_{*_s}(f)}b_1\cdots b_\ell\] and extend to $\A(\Lambda)$ by the
Leibniz rule.

In \cite{Ekholms1s2}, Ekholm and Ng prove the map $\partial$ has
degree $-1$ and is a differential, $\partial^2=0$.


\begin{figure}
  \labellist
  \small
  \pinlabel $1$ [B] at 174 284
  \pinlabel $2$ [B] at 174 260
  \pinlabel $3$ [B] at 174 235
  \pinlabel $4$ [B] at 174 211

  \pinlabel $\bar1$ [B] at 174 115
  \pinlabel $\bar2$ [B] at 174 81

  \pinlabel $1$ [B] at 779 284
  \pinlabel $2$ [B] at 779 260
  \pinlabel $3$ [B] at 779 235
  \pinlabel $4$ [B] at 779 211

  \pinlabel $\bar1$ [B] at 779 115
  \pinlabel $\bar2$ [B] at 779 81

  \pinlabel $1$ [B] at 1189 383
  \pinlabel $2$ [B] at 1189 346
  \pinlabel $3$ [B] at 1189 307
  \pinlabel $4$ [B] at 1189 270

  \pinlabel $\bar1$ [B] at 1189 117
  \pinlabel $\bar2$ [B] at 1189 81

  \pinlabel $1$ [B] at 1727 383
  \pinlabel $2$ [B] at 1727 345
  \pinlabel $3$ [B] at 1727 307
  \pinlabel $4$ [B] at 1727 270

  \pinlabel $\bar1$ [B] at 1727 118
  \pinlabel $\bar2$ [B] at 1727 79

  \pinlabel $1$ [b] at 1252 399
  \pinlabel $2$ [b] at 1250 302
  \pinlabel $3$ [t] at 1249 252

  \pinlabel $a_1$ [b] at 1400 388
  \pinlabel $a_2$ [b] at 1509 388
  \pinlabel $a_3$ [b] at 1441 281
  \pinlabel $a_4$ [b] at 1570 364
  \pinlabel $a_5$ [b] at 1600 401
  \pinlabel $a_6$ [b] at 1408 187
  \pinlabel $a_7$ [b] at 1448 205
  \pinlabel $a_8$ [t] at 1439 169
  \pinlabel $a_9$ [b] at 1481 190

  \endlabellist
  \includegraphics[width=\textwidth]{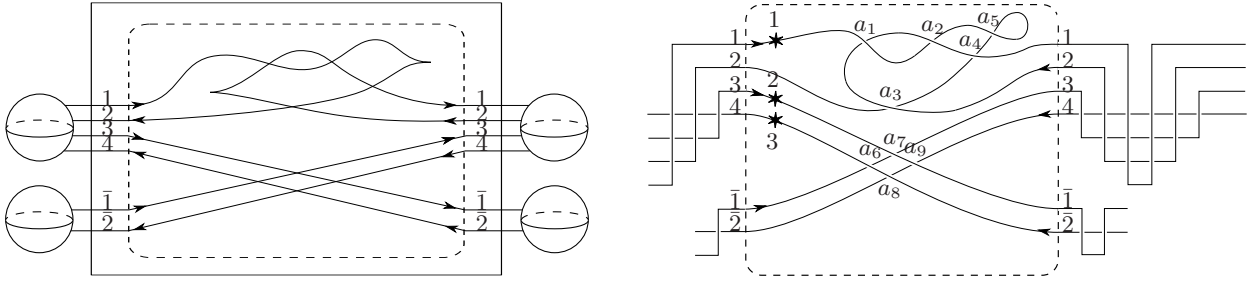}
  \caption{The left gives a Legendrian $xz$-diagram in $\#^2(S^1\times
    S^2)$ in Gompf standard form. The right gives the dip form of the
    normal form. Recall the labels on the crossings in the dips from
    Figure~\ref{fig:dipLabels} for the top $1$-handle and label the
    left crossing $\bar{b}_{12}$ and the right $\bar{c}_{12}$ in the
    dip of the bottom $1$-handle.}
  \label{fig:example}
\end{figure}

\begin{ex} \label{ex:dga} The following is the definition of the DGA
  $(\A(\Lambda),\partial)$ for the Legendrian link $\Lambda$ in Figure
  \ref{fig:example}. Here $\A(\Lambda)$ is generated by
  $a_1,\ldots,a_9,b_{ij},c^p_{ij}$ over
  $\integers[t_1^{\pm1},t_2^{\pm1},t_3^{\pm1}]$. We set $\lvert
  t_i\rvert=2r(\Lambda_i)=0$ for $i=1,2,3$. Define a Maslov potential
  $m$ on the strands near $x=0$ by
  \[\begin{array}{c|*{6}{c}}
    i&1&2&3&4&\bar1&\bar2\\
    \hline
    m(i)&2&1&0&-1&0&-1
  \end{array}\]
  Then we have the following gradings: $\lvert a_1\rvert=\lvert a_2\rvert=\lvert a_3\rvert=\lvert a_7\rvert=\lvert a_8\rvert=0$, $\lvert a_4\rvert=\lvert a_5\rvert=\lvert a_9\rvert=1$, $\lvert a_6\rvert=-1$,
  \[\begin{array}{c|*{7}{C{.15in}}}
    ij&12&13&14&23&24&34&\widebar{12}\\
    \hline \lvert b_{ij}\rvert&1&2&3&2&2&1&1\\
    \lvert c^0_{ij}\rvert&0&1&2&0&1&0&0
  \end{array}\]
  \[\hspace{-.2in}\begin{array}{*{6}{C{.15in}}}
    &&\multicolumn{4}{c}{j}\\
    &\multicolumn{1}{ c| }{\lvert c^1_{ij}\rvert}&1&2&3&4\\
    \cline{2-6}\multirow{4}{*}{i}&\multicolumn{1}{ c| }{1}&1&2&3&4\\
    &\multicolumn{1}{ c| }{2}&0&1&2&3\\
    &\multicolumn{1}{ c| }{3}&-1&0&1&2\\
    &\multicolumn{1}{ c| }{4}&-2&-1&0&1
  \end{array}
  \hspace{1in}
  \begin{array}{*{6}{C{.15in}}}
    &&\multicolumn{4}{c}{j}\\
    &\multicolumn{1}{ c| }{\lvert c^2_{ij}\rvert}&1&2&3&4\\
    \cline{2-6}\multirow{4}{*}{i}&\multicolumn{1}{ c| }{1}&3&4&5&6\\
    &\multicolumn{1}{ c| }{2}&2&3&4&5\\
    &\multicolumn{1}{ c| }{3}&1&2&3&4\\
    &\multicolumn{1}{ c| }{4}&0&1&2&3
  \end{array}
  \] where $\widebar{12}$ is the crossing of the strands in the bottom
  $1$-handle. Since $\lvert c^p_{ij}\rvert=2p-1+m(i)-m(j)$, we know
  $\lvert c^p_{ij}\rvert>0$ for $p>2$.

  For ease of notation, we will use $\bar{c}^p_{12}$ to denote
  $c^p_{\widebar{12}}$. We then have the following differentials:

  \begin{align*}
    \partial a_1&=\partial a_2=\partial a_3=\partial a_6=0\\
    \partial a_4&=(1+a_2a_1)a_3-t_1^{-1}a_2c^0_{12}\\
    \partial a_5&=1-a_1a_3+t_1^{-1}c^0_{12}\\
    \partial a_7&=t_2^{-1}t_3^{-1}c^0_{34}a_6\\
    \partial a_8&=a_6\bar{c}^0_{12}\\
    \partial a_9&=t_2^{-1}t_3^{-1}c^0_{34}a_8-a_7\bar{c}^0_{12}\displaybreak[2]\\[.2in]
    \partial b_{12}&=1+a_2a_1-c^0_{12}\\
    \partial b_{13}&=(1+a_2a_1)b_{23}+a_4(t_2c^0_{23}a_7+t_3^{-1}c^0_{24}a_6)-t_1^{-1}a_2(t_2c^0_{13}a_7+t_3^{-1}c^0_{14}a_6)-c^0_{13}+b_{12}c^0_{23}\\
    \partial b_{14}&=(1+a_2a_1)b_{24}-[a_4(t_2c^0_{23}a_7+t_3^{-1}c^0_{24}a_6)-t_1^{-1}a_2(t_2c^0_{13}a_7+t_3^{-1}c^0_{14}a_6)]b_{34}\\
    &\quad+(a_4c^0_{23}-t_1^{-1}a_2c^0_{13})t_2a_9+(a_4c^0_{24}-t_1^{-1}a_2c^0_{14})t_3^{-1}a_8-c^0_{14}+b_{12}c^0_{24}-b_{13}c^0_{34}\\
    \partial b_{23}&=-a_3(t_2c^0_{23}a_7+t_3^{-1}c^0_{24}a_6)-c^0_{23}\\
    \partial b_{24}&=-a_3(t_2c^0_{23}a_7+t_3^{-1}c^0_{24}a_6)b_{34}-t_3^{-1}a_3c^0_{24}a_8-c^0_{24}+b_{23}c^0_{34}-t_2a_3c^0_{23}a_9\\
    \partial b_{34}&=\bar{c}^0_{12}-c^0_{34}\displaybreak[3]\\[.2in]
    \partial \bar{b}_{12}&=t_2^{-1}t_3^{-1}c^0_{34}-\bar{c}^0_{12}\displaybreak[2]\\[.2in]
    \partial
    c^p_{ij}&=\delta_{ij}\delta_{1p}+\sum_{\ell=0}^p\sum_{m=1}^4(-1)^{\lvert
      c^\ell_{im}\rvert+1}c^\ell_{im}c^{p-\ell}_{mj}\\[.2in]
    \partial
    \bar{c}^p_{ij}&=\delta_{ij}\delta_{1p}+\sum_{\ell=0}^p\sum_{m=1}^2(-1)^{\lvert
      \bar{c}^\ell_{im}\rvert+1}\bar{c}^\ell_{im}\bar{c}^{p-\ell}_{mj}
  \end{align*}
\end{ex}

\begin{defn}
  Let $(\A,\partial)$ be a semifree DGA over $R$ generated by
  $\{a_i\vert i\in I\}$. Let $J$ be a countable (possibly finite)
  index set. A {\bf stabilization} of $(\A,\partial)$ is the semifree
  DGA $(S(\A),\partial)$, where $S(\A)$ is the tensor algebra over $R$
  generated by $\{a_i\vert i\in I\}\cup\{\alpha_j\vert j\in
  J\}\cup\{\beta_j\vert j\in J\}$ and the grading on $a_i$ is
  inherited from $\A$ and $\lvert \alpha_j\rvert=\lvert
  \beta_j\rvert+1$ for all $j\in J$. Let the differential on $S(\A)$
  agree with the differential on $\A\subset S(\A)$, define
  \[\partial(\alpha_j)=\beta_j\text{ and }\partial(\beta_j)=0\]
  for all $j\in J$, and extend by the Leibniz rule.
\end{defn}

\begin{defn}
  Two semifree DGAs $(\A,\partial)$ and $(\A',\partial')$ are
  {\bf stable tame isomorphic} if some stabilization of
  $(\A,\partial)$ is tamely isomorphic (see \cite{Ekholms1s2}) to
  some stabilization of $(\A',\partial')$.
\end{defn}

\begin{thm}[\cite{Ekholms1s2} Theorem 2.18]
  Let $\Lambda$ and $\Lambda'$ be Legendrian isotopic Legendrian links
  in $\#^k(S^1\times S^2)$ in normal form. Let
  $(\A(\Lambda),\partial)$ and $(\A(\Lambda'),\partial')$ be the
  semifree DGAs over $R=\integers[t^{\pm1}_1,\ldots,t^{\pm1}_s]$
  associated to the diagrams $\pi_{xy}(\Lambda)$ and
  $\pi_{xy}(\Lambda')$, which are in normal form. Then
  $(\A(\Lambda),\partial)$ and $(\A(\Lambda'),\partial')$ are stable
  tame isomorphic.
\end{thm}

\begin{defn} Let $F$ be a field. An {\bf augmentation} of
  $(\A(\Lambda),\partial)$ to $F$ is a ring map
  $\epsilon:\A(\Lambda)\to F$ such that $\epsilon\circ\partial=0$ and
  $\epsilon(1)=1$. If $\rho\vert2r(\Lambda)$ and $\epsilon$ is
  supported on generators of degree divisible by $\rho$, then
  $\epsilon$ is $\rho$-graded. In particular, if $\rho=0$, we say it
  is {\bf graded} and if $\rho=1$, we say if is {\bf ungraded}. We
  call a generator $a$ {\bf augmented} if $\epsilon(a)\neq0$.
\end{defn}

\begin{ex}
  Recalling the DGA of the Legendrian link in Figure \ref{fig:example}
  of Example~\ref{ex:dga}, given a field $F$, one can check that any
  graded augmentation $\epsilon$ to $F$ satisfies the following:
  $\epsilon(t_1)=-1$, $\epsilon(t_3)=\epsilon(t_2)^{-1}$ where
  $\epsilon(t_2)\neq0$, $\epsilon(b_{ij})=\epsilon(\bar{b}_{12})=0$,
  and for $a,b,c,d,e,f\in F$ such that $1+ab,d,e\neq0$
  
  \[\begin{array}{c|*{9}{C{.1in}}}
    i&1&2&3&4&5&6&7&8&9\\
    \hline \epsilon(a_i)&a&b&-b&0&0&0&c&c&0
  \end{array}
  \hspace{1in}\begin{array}{c|*{7}{c}}
    ij&12&13&14&23&24&34&\widebar{12}\\
    \hline\epsilon(c^0_{ij})&1+ab&0&0&0&0&d&d
  \end{array}\]
  \[\hspace{-.2in}\begin{array}{*{6}{c}}
    &&\multicolumn{4}{c}{j}\\
    &\multicolumn{1}{ c| }{\lvert c^1_{ij}\rvert}&1&2&3&4\\
    \cline{2-6}\multirow{4}{*}{i}&\multicolumn{1}{ c| }{1}&0&0&0&0\\
    &\multicolumn{1}{ c| }{2}&e&0&0&0\\
    &\multicolumn{1}{ c| }{3}&0&f&0&0\\
    &\multicolumn{1}{ c| }{4}&0&0&(1+ab)d^{-1}e&0
  \end{array}
  \hspace{1.5in}
  \begin{array}{*{6}{c}}
    &&\multicolumn{4}{c}{j}\\
    &\multicolumn{1}{ c| }{\lvert c^2_{ij}\rvert}&1&2&3&4\\
    \cline{2-6}\multirow{4}{*}{i}&\multicolumn{1}{ c| }{1}&0&0&0&0\\
    &\multicolumn{1}{ c| }{2}&0&0&0&0\\
    &\multicolumn{1}{ c| }{3}&0&0&0&0\\
    &\multicolumn{1}{ c| }{4}&-(1+ab)d^{-1}f&0&0&0
  \end{array}
  \]
\end{ex}

Note that any augmentation of a stabilization $S(\A)$ restricts to an
augmentation of the smaller algebra $\A$ and any augmentation of the
algebra $\A$ extends to an augmentation of the stabilization $S(\A)$
where the augmentation sends $\beta_j$ to $0$ and $\alpha_j$ to an
arbitrary element of $F$ if $\rho\vert\lvert \alpha_j\rvert$ and $0$
otherwise for all $j\in J$.

\subsection{Normal rulings in $\#^k(S^1\times S^2)$}

In this section, we extend the definition of a normal ruling from
Legendrian links in $\reals^3$ to Legendrian links in $\#^k(S^1\times
S^2)$. We formulate the definition similarly to how Lavrov and
Rutherford \cite{LavrovSolidTorus} define normal rulings in the case
of Legendrian links in the solid torus.

Consider the tangle portion of the $\pi_{xz}(\Lambda)$ diagram in
normal form of a Legendrian link $\Lambda\subset\#^k(S^1\times S^2)$. A
normal ruling can be viewed locally as a decomposition of
$\pi_{xz}(\Lambda)$ into pairs of paths.

Let $C\subset S^1$ be the set of $x$-coordinates of crossings and
cusps of $\pi_{xz}(\Lambda)$ where $S^1=[0,A]/\{0=A\}$. We can write
\[S^1\backslash C=\coprod_{\ell=1}^MI_\ell\] where $I_\ell$ is an open
interval (or all of $S^1$) for each $\ell$. We will use the convention
that $I_0=I_M$ and the $I_\ell$ are ordered $I_0,\ldots,I_M$ from
$x=0$ to $x=A$ (from left to right in the $xz$-diagram) so that
$I_{\ell-1}$ appears to the left of (has lower $x$-coordinates than)
$I_\ell$. Note that $(I_\ell\times[-M,M])\cap\pi_{xz}(\Lambda)$
consists of some number of nonintersecting components which project
homeomorphically onto $I_\ell$. We call these components {\bf strands}
of $\pi_{xz}(\Lambda)$ and number them from top to bottom by
$1,\ldots,N(\ell)$. For each $\ell$, choose a point $x_\ell\in
I_\ell$.

\begin{defn}\label{defn:normalRuling}
  A {\bf normal ruling} of $\pi_{xz}(\Lambda)$ is a sequence of
  involutions $\sigma=(\sigma_1,\ldots,\sigma_M)$,
  \begin{gather*}
    \sigma_m:\{1,\ldots,N(m)\}\to\{1,\ldots,N(m)\}\\
    (\sigma_m)^2=id,
  \end{gather*}
satisfying:
  \begin{enumerate}
  \item Each $\sigma_m$ is fixed-point-free.
  \item If the strands above $I_m$ labeled $\ell$ and $\ell+1$ meet at
    a left cusp in the interval $(x_{m-1},x_m)$, then
    \[\sigma_m(i)=\begin{cases}
      \ell+1&\text{if }i=\ell,\\
      \sigma_{m-1}(i)&\text{if }i<\ell,\\
      \sigma_{m-1}(i-2)&\text{if }i>\ell+1.
    \end{cases}\] And a similar condition at right cusps.
  \item \label{cond:switch} If strands above $I_m$ labeled $\ell$ and
    $\ell+1$ meet at a crossing on the interval $(x_{m-1},x_m)$, then
    $\sigma_{m-1}(\ell)\neq\ell+1$ and either
    \begin{itemize}
    \item
      $\sigma_m=(\ell\quad\ell+1)\circ\sigma_{m-1}\circ(\ell\quad\ell+1)$
      where $(\ell\quad\ell+1)$ denotes transposition or
    \item $\sigma_m=\sigma_{m-1}$.
    \end{itemize}
    When the second case occurs, we call the crossing {\bf
      switched}. We say the normal ruling is {\bf $\rho$-graded} if
    $\rho\big\vert\lvert c\rvert$ for all switched crossings $c$.
  \item \label{cond:normal} (Normality condition) If there is a
    switched crossing on the interval $(x_{m-1},x_m)$, then one of the
    following holds:
    \begin{itemize}
    \item $\sigma_m(\ell+1)<\sigma_m(\ell)<\ell<\ell+1$
    \item $\sigma_m(\ell)<\ell<\ell+1<\sigma_m(\ell)$
    \item $\ell<\ell+1<\sigma_m(\ell+1)<\sigma_m(\ell)$
    \end{itemize}
  \item \label{cond:sameHandle} Near $x=0$ and $x=A$, both the strand
    with label $\ell$ and $\sigma_0(\ell)$ must go through the same
    $1$-handle, in other words, there exists $p$ such that
    $\sum_{i=1}^{p-1}N_i<\ell,\sigma_0(\ell)\leq\sum_{i=1}^pN_i$.
  \end{enumerate}
\end{defn}

The final condition is the only condition which is different from how
normal rulings are defined in \cite{LavrovSolidTorus} for the case of
solid torus knots. This condition ensures the ruling ``behaves well''
with the $1$-handles.

\begin{rmk}
  As in \cite{LavrovSolidTorus}, one can equivalently see normal
  rulings as pairings of strands in the $xz$-diagram with certain
  conditions. Here we think of strands $i$ and $j$ being paired for
  $x_{m-1}\leq x\leq x_m$ if $\sigma_m(i)=j$. In this way, we can
  cover the $xz$-diagram with pairs of paths which have monotonically
  increasing $x$-coordinate. Note that if a path goes all the way from
  $x=0$ to $x=A$, it may end up on a different strand than it started,
  but strand $i$ is paired with strand $j$ at $x=0$ if and only if
  they are paired at $x=A$. Condition~\ref{cond:sameHandle} also
  specifies that the paired strands must go through the same
  $1$-handle. The conditions mentioned above are as follows: Paired
  paths can only meet at a cusp. This also means that at a crossing,
  the crossings strands must be paired with other strands. These {\bf
    companion strands} can either lie above or below the
  crossing. Conditions~\ref{cond:switch} and \ref{cond:normal} specify
  that near a crossing the pairings must be one of those depicted in
  Figure~\ref{fig:config}.
\end{rmk}

\begin{ex}
  Figure~\ref{fig:exRulings} gives the normal rulings of the Legendrian
  link from Example~\ref{ex:dga}.
\end{ex}

\begin{figure}
  \labellist
  \small\hair 2pt
  \pinlabel $(a)$ [t] at 82 267
  \pinlabel $(b)$ [t] at 339 267
  \pinlabel $(c)$ [t] at 593 267
  \pinlabel $(d)$ [t] at 82 -20
  \pinlabel $(e)$ [t] at 339 -20
  \pinlabel $(f)$ [t] at 593 -20
  \endlabellist
  \includegraphics[width=.6\textwidth]{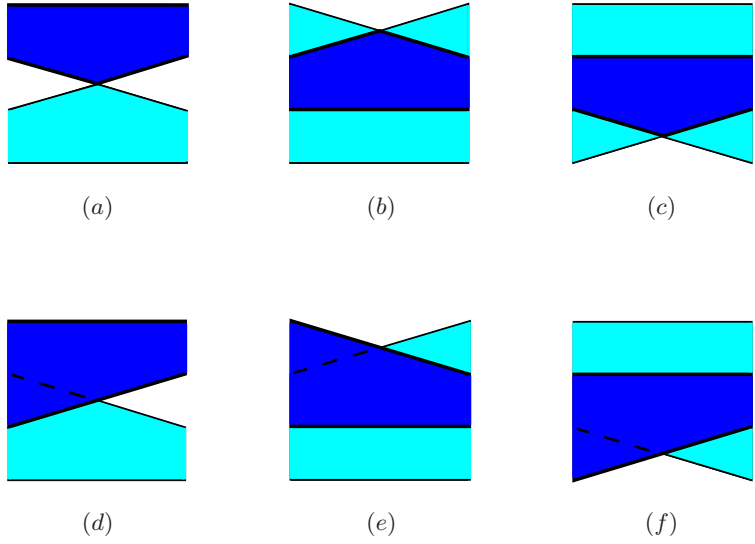}
  \vspace{.25in}
  \caption{These, along with vertical reflections of (d), (e), and
    (f), are all possible configurations of a normal ruling near a
    crossing. The top row contains all possible configurations for
    switched crossings in a normal ruling. (This figure is taken from
    \cite{Leverson}.)}
  \label{fig:config}
\end{figure}

\begin{figure}
  \includegraphics[width=\textwidth]{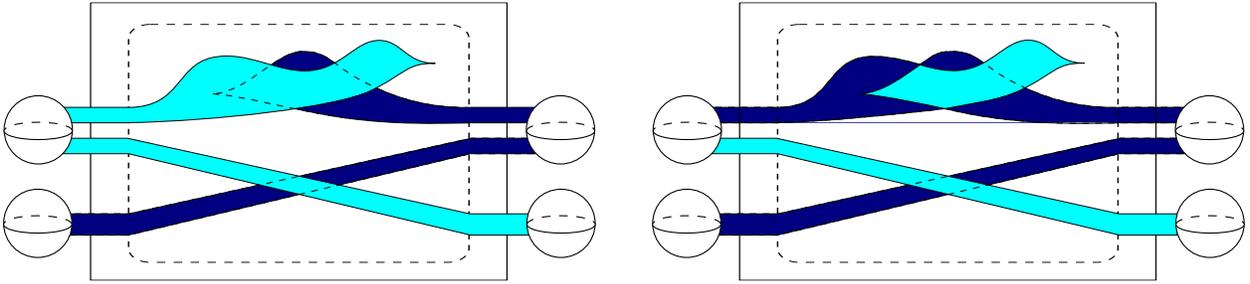}
  \vspace{.25in}
  \caption{These are the two normal rulings of the Legendrian link of
    Example~\ref{ex:dga} seen in Figure~\ref{fig:example}.}
  \label{fig:exRulings}
\end{figure}

Similarly to $\reals^3$, we can define a $\rho$-graded ruling
polynomial.

\begin{defn}
  If $m$ is a $\integers/\rho$-valued Maslov potential for a
  Legendrian link $\Lambda$, then the {\bf $\rho$-graded ruling
    polynomial} of $\Lambda$ with respect to $m$ is
  \[R^\rho_{(\Lambda,m)}=\sum_\sigma z^{j(\sigma)},\] where the sum is
  over all $\rho$-graded normal rulings of $\Lambda$ and
  \[j(\sigma)=\text{\# switches}-\text{\# right cusps}.\]
\end{defn}

Note that in the case where $\Lambda$ is a knot, the ruling polynomial
does not depend on the Maslov potential. Restated from the introduction:

\begin{thmRulingPoly}
  The $\rho$-graded ruling polynomial $R^\rho_{(\Lambda,m)}$ with
  respect to the Maslov potential $m$ (which changes under Legendrian
  isotopy) is a Legendrian isotopy invariant.
\end{thmRulingPoly}

\begin{proof}
  By Gompf~\cite{GompfHandlebody}, any Legendrian link in
  $\#^k(S^1\times S^2)$ can be represented by an $xz$-diagram in Gompf
  standard form and two such $xz$-diagrams represent links that are
  Legendrian isotopic if and only if they are related by a sequence of
  Legendrian Reidemeister moves of the $xz$-diagram of the tangle
  inside $[0,A]\times[-M,M]$ and three additional moves, which we
  will, following the nomenclature of \cite{Ekholms1s2}, call Gompf
  moves 4, 5, and 6 (see Figure~\ref{fig:gompfMoves}). By
  \cite{ChekanovFronts}, we know the ruling polynomial is invariant
  under Legendrian isotopy of the tangle, so we need only show it is
  invariant under Gompf moves 4, 5, and 6.

  \begin{figure}
    \labellist
    \small\hair 2pt 
    \pinlabel $4:$ [r] at -10 1216
    \pinlabel ${\large\Lambda}$ at 844 1216

    \pinlabel ${\large\Lambda}$ at 2873 1216

    \pinlabel $5:$ [r] at -10 701
    \pinlabel ${\large\Lambda}$ at 844 701

    \pinlabel ${\large\Lambda}$ at 2873 701

    \pinlabel $6:$ [r] at -10 182

    \pinlabel ${\large\Lambda}$ at 2873 182
    \endlabellist
    \includegraphics[width=.9\textwidth]{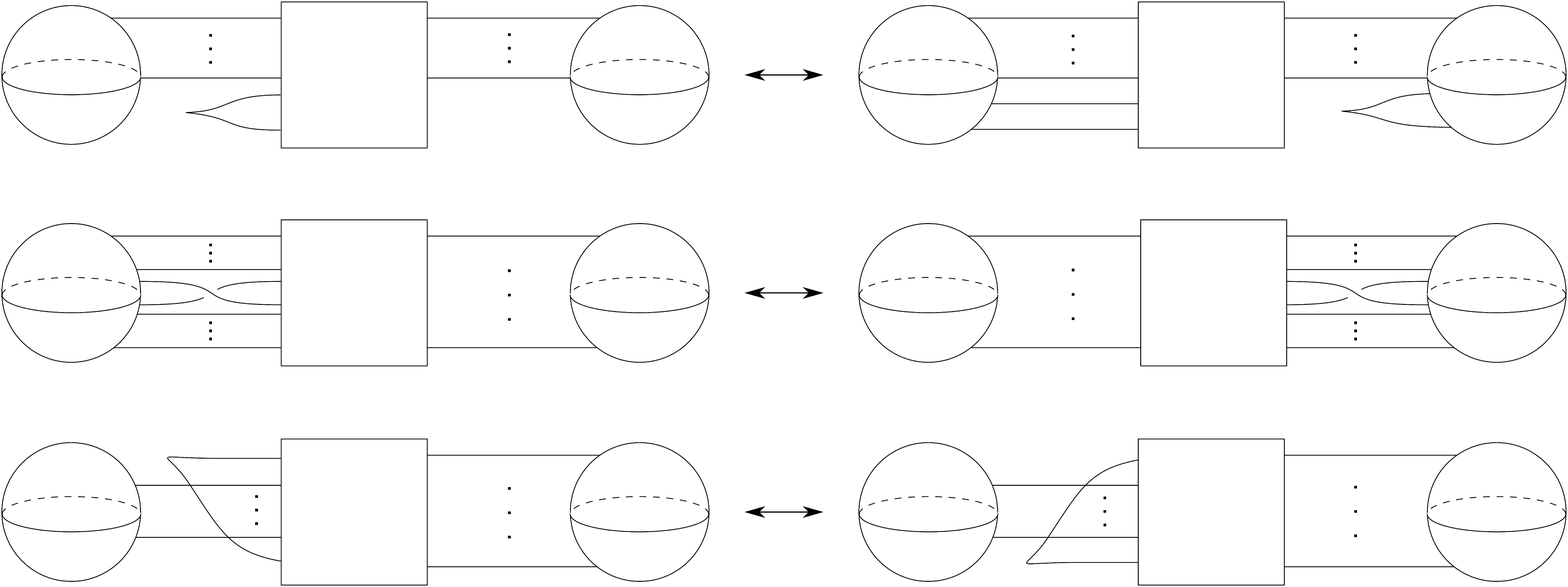}
    \vspace{.25in}
    \caption{Gompf moves 4, 5, and 6.}
    \label{fig:gompfMoves}
  \end{figure}

  Gompf moves 4 and 5 clearly do not change the ruling polynomial. For
  Gompf move 6, note that any normal ruling cannot pair a strand going
  through the $1$-handle with one of the strands incident to the
  cusp. Instead, the ruling must pair the two strands incident to the
  left cusp and not have any switches in the portion of the diagram
  depicted in Figure~\ref{fig:gompfMoves}, thus the ruling polynomial
  does not change.
\end{proof}

\begin{ex}
  The normal rulings for the Legendrian link from Example~\ref{ex:dga}
  are given in Figure~\ref{fig:exRulings}. Thus the ruling polynomial
  is
  \[R_{\Lambda}=z^{-1}+z.\]
\end{ex}

\subsection{Legendrian links in $\reals^3$}
\label{sec:linksR3}

The classical invariants for Legendrian isotopy classes of knots in
$\reals^3$ are: topological knot type, Thurston-Bennequin number, and
rotation number (see \cite{EtnyreLegendrianTrans}). The {\bf
  Thurston-Bennequin number} of a knot measures the self-linking of a
Legendrian knot $\Lambda$. Given a push off $\Lambda'$ of $\Lambda$ in
a direction tangent to the contact structure, then $tb(\Lambda)$ is
the linking number of $\Lambda$ and $\Lambda'$. Given the
$xz$-projection of $\Lambda$,
\[tb(\Lambda)=\text{writhe}(\Lambda)-\frac12(\text{number of
  cusps}).\] The {\bf rotation number} $r(\Lambda)$ of an oriented
Legendrian knot $\Lambda$ is the rotation of its tangent vector field
with respect to any global trivialization. (This definition agrees
with the definition of the rotation number of a path given earlier.)
Given the $xz$-projection of $\Lambda$,
\[r(\Lambda)=\frac12(\text{number of down cusps}-\text{number of up
  cusps}).\] Given a Legendrian link
$\Lambda=\Lambda_1\coprod\cdots\coprod\Lambda_n$, we define
$tb_i=tb(\Lambda_i)$ and $r_i=r(\Lambda_i)$ for $1\leq i\leq n$ and
define
\[r(\Lambda)=\gcd(r_1,\ldots,r_n).\]

\subsection{Satellites, the DGA, and augmentations in
  $\reals^3$}\label{sec:satellite}

This section gives the results and notation for Legendrian links in
$\reals^3$ necessary to prove Theorem~\ref{thm:main}.

We will first extend the idea of satelliting a knot in $J^1(S^1)$ to
an unknot (see \cite{NgSatellites}) to satelliting each $1$-handle of
a knot in $\#^k(S^1\times S^2)$ around a twice stabilized unknot.

\begin{defn}
  Given the $xy$- or $xz$-diagram for a Legendrian link $\Lambda$ in
  $\#^k(S^1\times S^2)$, {\bf satellited $\Lambda$} is denoted
  $S(\Lambda)$, the $xy$-diagram of which is depicted in
  Figure~\ref{fig:exSatellitedxy} and the $xz$-diagram of a Legendrian
  isotopic link of which is depicted in
  Figure~\ref{fig:exSatellitedxz} for the Legendrian link from
  Figure~\ref{fig:example}. Label the crossings as indicated, where
  $i\leq j$ and label the base points in $S(\Lambda)$ as they are
  labeled in $\Lambda$. Note that the $xy$- or $xz$-diagram of
  $\Lambda$ defines $S(\Lambda)$ up to Legendrian isotopy.
\end{defn}

\begin{figure}
  \labellist
  \small
  \pinlabel $1$ [b] at 750 1954
  \pinlabel $2$ [b] at 750 1889
  \pinlabel $3$ [b] at 750 1824
  \pinlabel $4$ [b] at 750 1758

  \pinlabel $\bar1$ [b] at 750 1488
  \pinlabel $\bar2$ [b] at 750 1423

  \pinlabel $1$ [b] at 1708 1963
  \pinlabel $2$ [b] at 1708 1898
  \pinlabel $3$ [b] at 1708 1829
  \pinlabel $4$ [b] at 1708 1767

  \pinlabel $\bar1$ [b] at 1708 1498
  \pinlabel $\bar2$ [b] at 1708 1432

  \pinlabel $t_1$ [b] at 879 1993
  \pinlabel $t_2$ [B] at 879 1840
  \pinlabel $t_3$ [t] at 878 1755

  \pinlabel $a_1$ [b] at 1129 1975
  \pinlabel $a_2$ [b] at 1321 1975
  \pinlabel $a_3$ [b] at 1201 1784
  \pinlabel $a_4$ [b] at 1430 1929
  \pinlabel $a_5$ [b] at 1490 1998
  \pinlabel $a_6$ [b] at 1144 1618
  \pinlabel $a_7$ [b] at 1212 1653
  \pinlabel $a_8$ [b] at 1200 1589 
  \pinlabel $a_9$ [b] at 1273 1625

  \pinlabel $b_{ij}$ at 1764 1660
  \pinlabel $c_{ij}$ at 2043 1665
  \pinlabel $\bar{b}_{12}$ [tr] at 1741 1431
  \pinlabel $\bar{c}_{12}$ [tl] at 1804 1432

  \pinlabel $d_{ji}$ [l] at 2748 1571
  \pinlabel $e_{ij}$ [r] at 2254 1100
  \pinlabel $f_{ji}$ [l] at 2720 712
  \pinlabel $g_{ij}$ [r] at 7 699
  \pinlabel $h_{ji}$ [l] at 450 1131
  \pinlabel $q_{ij}$ [r] at 8 1592

  \pinlabel $\bar{d}_{ji}$ [l] at 2120 1307
  \pinlabel $\bar{e}_{12}$ [r] at 1825 1116
  \pinlabel $\bar{f}_{ji}$ [l] at 2118 934
  \pinlabel $\bar{g}_{12}$ [r] at 640 937
  \pinlabel $\bar{h}_{ji}$ [l] at 899 1108
  \pinlabel $\bar{q}_{12}$ [r] at 622 1329
  \endlabellist
  \includegraphics[width=\textwidth]{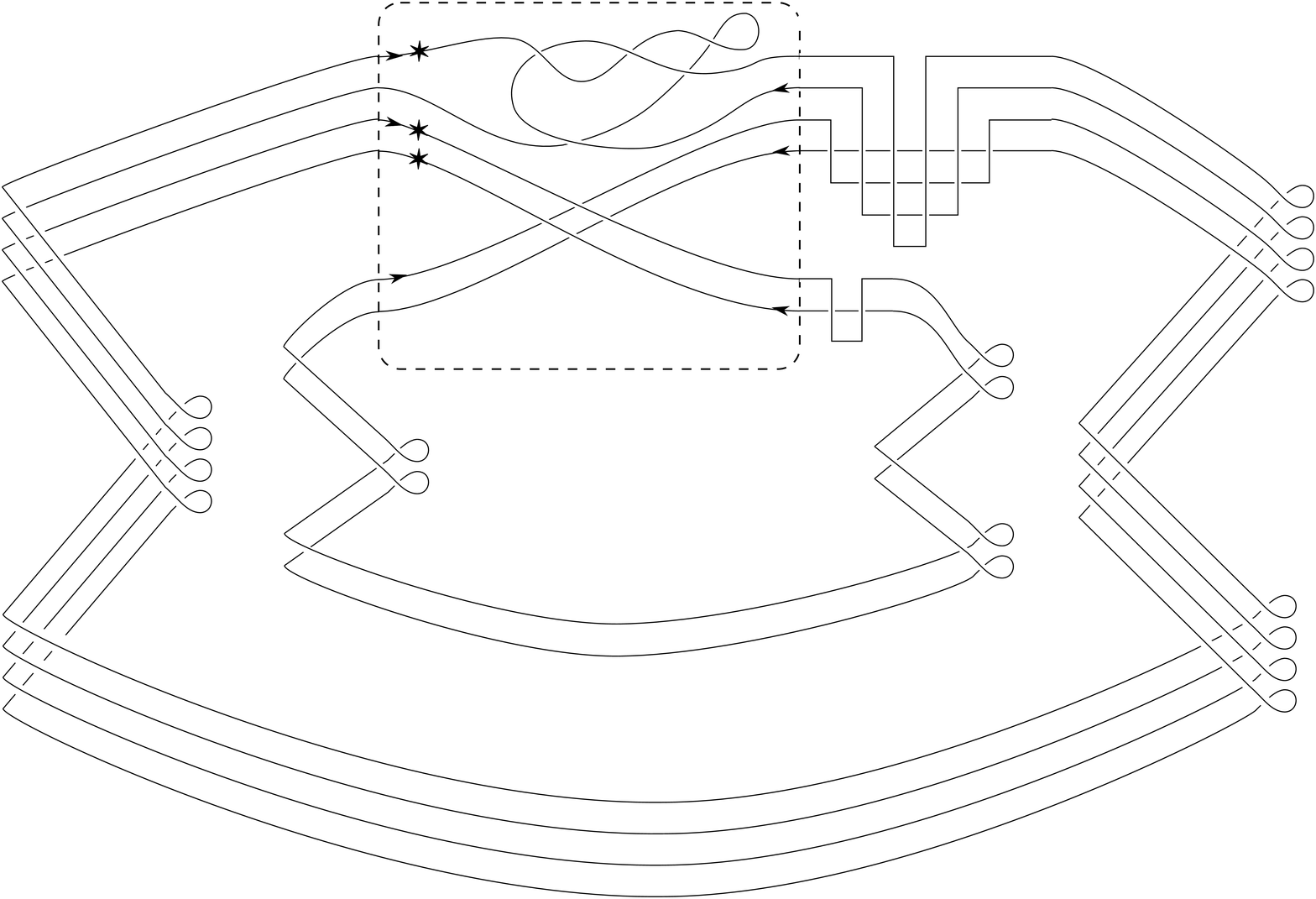}
  \caption{The $xy$-projection of the satellited link
    $S(\Lambda)$. The crossings in the $c_{ij}$-, $b_{ij}$-,
    $\bar{c}_{ij}$, and $\bar{b}_{ij}$-lattices are labeled as in
    Figure~\ref{fig:dipLabels}. The crossings in the
    $d,e,f,g,h,q$-lattices are labelled according to
    Figure~\ref{fig:satelliteLabels}.}
  \label{fig:exSatellitedxy}
\end{figure}

\begin{figure}
  \labellist
  \small

  \pinlabel $1$ [b] at 196 381
  \pinlabel $2$ [b] at 255 381
  \pinlabel $3$ [b] at 312 381
  \pinlabel $4$ [b] at 372 381

  \pinlabel $1$ [t] at 395 7
  \pinlabel $2$ [t] at 329 7
  \pinlabel $3$ [t] at 262 7
  \pinlabel $4$ [t] at 195 7

  \pinlabel $e_{12}$ [b] at 142 268
  \pinlabel $e_{13}$ [b] at 175 237
  \pinlabel $e_{14}$ [b] at 208 205
  \pinlabel $e_{23}$ [b] at 143 201
  \pinlabel $e_{24}$ [b] at 175 170
  \pinlabel $e_{34}$ [b] at 144 136

  \pinlabel $1$ [b] at 892 381
  \pinlabel $2$ [b] at 840 381
  \pinlabel $3$ [b] at 788 381
  \pinlabel $4$ [b] at 738 381

  \pinlabel $1$ [t] at 724 7
  \pinlabel $2$ [t] at 780 7
  \pinlabel $3$ [t] at 837 7
  \pinlabel $4$ [t] at 893 7

  \pinlabel $d_{11}$ [b] at 971 303
  \pinlabel $d_{21}$ [b] at 937 268
  \pinlabel $d_{31}$ [b] at 910 238
  \pinlabel $d_{41}$ [b] at 883 207
  \pinlabel $d_{22}$ [b] at 970 237
  \pinlabel $d_{32}$ [b] at 937 203
  \pinlabel $d_{42}$ [b] at 910 170
  \pinlabel $d_{33}$ [b] at 971 171
  \pinlabel $d_{43}$ [b] at 937 137
  \pinlabel $d_{44}$ [b] at 968 107
  \endlabellist
  \includegraphics[width=.8\textwidth]{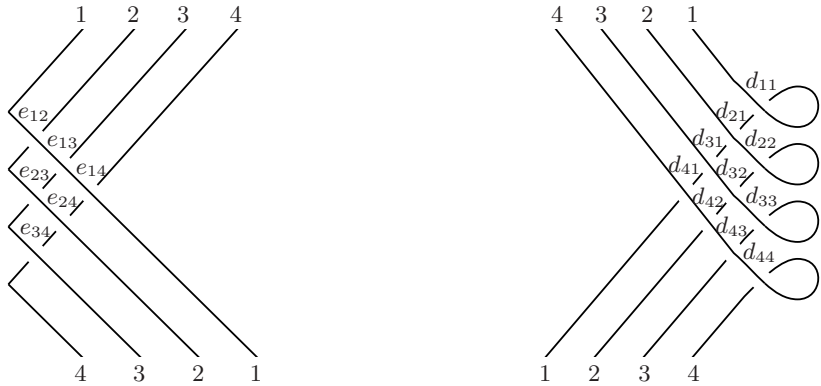}
  \caption{The labels for the crossings in the $e$- and $d$-lattices
    of the satellited link $S(\Lambda)$ as seen in
    Figure~\ref{fig:exSatellitedxy}. The $f$- and $h$-lattices are
    analogous to the $d$-lattice. The $g$- and $q$-lattices are
    analogous to the $e$-lattice.}
  \label{fig:satelliteLabels}
\end{figure}

\begin{figure}
  \labellist
  \small
  \pinlabel $1$ [B] at 543 824
  \pinlabel $2$ [B] at 543 801
  \pinlabel $3$ [B] at 543 779
  \pinlabel $4$ [B] at 543 757

  \pinlabel $1'$ [B] at 543 658
  \pinlabel $2'$ [B] at 543 628

  \pinlabel $1$ [B] at 1210 825
  \pinlabel $2$ [B] at 1210 802
  \pinlabel $3$ [B] at 1210 780
  \pinlabel $4$ [B] at 1210 757

  \pinlabel $1'$ [B] at 1210 664
  \pinlabel $2'$ [B] at 1210 633

  \pinlabel $a_1$ [b] at 833 888
  \pinlabel $a_2$ [b] at 927 891
  \pinlabel $a_3$ [b] at 865 825
  \pinlabel $a_4$ [b] at 989 855
  \pinlabel $a_5$ [l] at 1074 891
  \pinlabel $a_6$ [b] at 809 710
  \pinlabel $a_7$ [b] at 874 724
  \pinlabel $a_8$ [b] at 874 696
  \pinlabel $a_9$ [b] at 939 710

  \pinlabel $d_{ji}$ [l] at 1737 653
  \pinlabel $e_{ij}$ [r] at 1342 433
  \pinlabel $f_{ji}$ [l] at 1736 232
  \pinlabel $g_{ij}$ [r] at 8 232
  \pinlabel $h_{ji}$ [l] at 393 434
  \pinlabel $q_{ij}$ [r] at 8 652

  \pinlabel $d'_{ji}$ [l] at 1342 541
  \pinlabel $e'_{ij}$ [r] at 1153 431
  \pinlabel $f'_{ji}$ [l] at 1342 318
  \pinlabel $g'_{ij}$ [r] at 402 321
  \pinlabel $h'_{ji}$ [l] at 588 430
  \pinlabel $q'_{ij}$ [r] at 402 542

  \endlabellist
  \includegraphics[width=\textwidth]{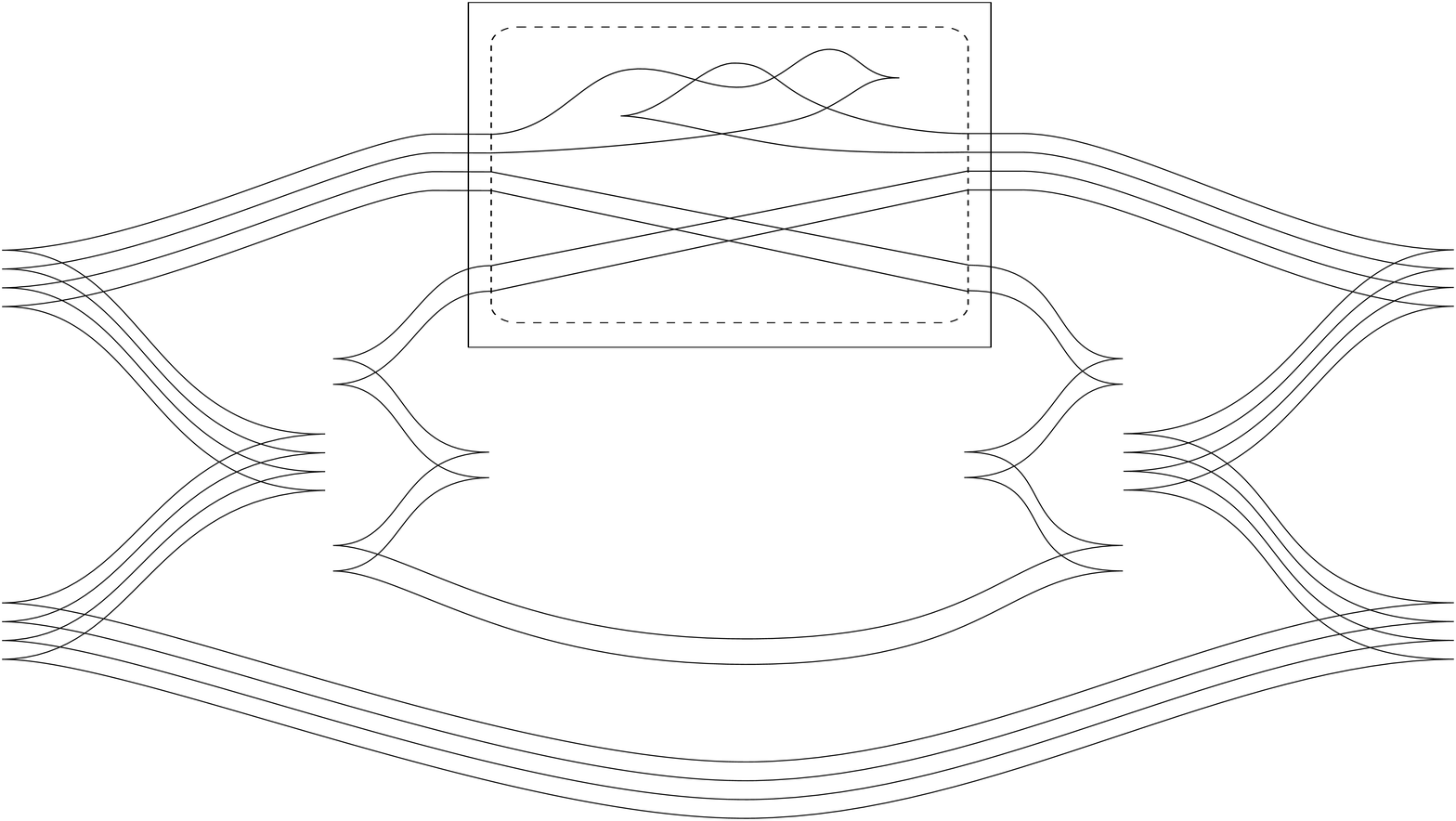}
  \caption{The $xz$-projection of a link which is Legendrian isotopic
    to the satellited link $S(\Lambda)$.}
  \label{fig:exSatellitedxz}
\end{figure}

\begin{rmk}
  The Chekanov-Eliashberg DGA was originally defined on Legendrian
  links in $(\reals^3,dz-ydx)$ (see
  \cite{Chekanov},\cite{SabloffAug}). Note that the same DGA results
  from defining the DGA as we did in $\#^k(S^1\times S^2)$ where
  $k=0$.
\end{rmk}

\subsection{Dips} \label{sec:dips}
Dips will be defined analogously to those defined in \cite{Leverson}.

Given a diagram $\pi_{xy}(\Lambda)$ in normal form which is the result
of resolution, we construct a {\bf dip} in the vertical slice of the
diagram between two crossings, a crossing and a cusp, or two cusps, by
a sequence of Reidemeister II moves, as seen in Figure~\ref{fig:dips}
in the $xz$-projection and $xy$-projection. From the $xz$-projection,
it is clear that the diagram with the dip is Legendrian isotopic to
the original diagram. To construct a dip, number the $N$ strands from
top to bottom. Using a type II Reidemeister move, push strand $N-1$ over
strand $N$, then strand $N-2$ over strand $N-1$, then strand $N-2$ over
strand $N$, and so on. In this way, strand $i$ is pushed over strand
$j$ in anti-lexicographic order.

\begin{figure}
  \labellist
  \small\hair 1pt
  \pinlabel $4$ [r] at 320 99
  \pinlabel $3$ [r] at 320 132
  \pinlabel $2$ [r] at 320 164
  \pinlabel $1$ [r] at 320 198

  \pinlabel $b_{14}$ [tr] at 425 99
  \pinlabel $b_{13}$ [tr] at 425 67
  \pinlabel $b_{12}$ [tr] at 425 35
  \pinlabel $b_{24}$ [tr] at 393 99
  \pinlabel $b_{23}$ [tr] at 393 67
  \pinlabel $b_{34}$ [tr] at 362 99

  \pinlabel $c_{14}$ [tl] at 472 99
  \pinlabel $c_{13}$ [tl] at 472 67
  \pinlabel $c_{12}$ [tl] at 472 35
  \pinlabel $c_{24}$ [tl] at 505 99
  \pinlabel $c_{23}$ [tl] at 505 67
  \pinlabel $c_{34}$ [tl] at 537 99

  \endlabellist

  \includegraphics[width=4in]{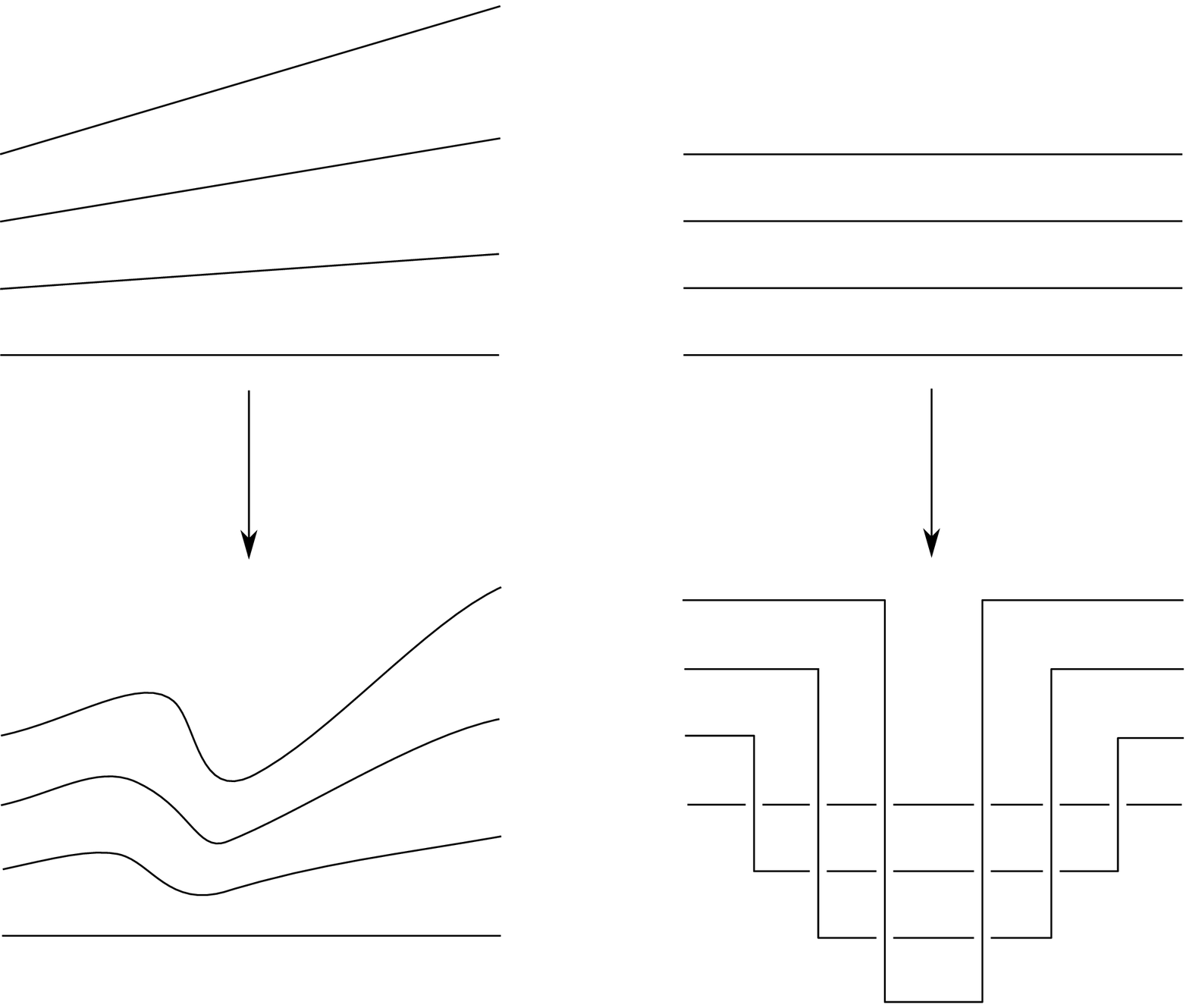}
  \caption{The left diagram gives the modification of the $xz$-diagram
    when creating a dip. The right diagram gives the modification of
    the $xy$-diagram. (This figure is taken from \cite{Leverson}.)}
  \label{fig:dips}
\end{figure}

Given an $xy$-diagram for a link $\Lambda\subset\reals^3$ in normal
form, where all crossings and resolutions of left cusps having
distinct $x$-coordinates, the {\bf dipped diagram} $D(\Lambda)$ is the
result of adding a dip between each pair of crossings or resolution of
a cusp and crossing. For each Reidemeister II move, we have two new
generators. Call the left crossing $b_{ij}$ and the right crossing
$c_{ij}$ if strands $i<j$ cross. One can check that $\lvert
b_{ij}\rvert=m(j)-m(i)$ and since $\partial$ lowers degree by $1$, we
know $\lvert c_{ij}\rvert=\lvert b_{ij}\rvert-1$.

While dipped diagrams have many more crossings than the original link
diagram, the differential $\partial$ on $\A(D(\Lambda))$ is generally
much simpler. In fact, a {\bf totally augmented disk} (a disk from the
definition of the differential of the DGA where all crossings at
corners are augmented), cannot ``go through'' or ``span'' more than
one dip.

\subsection{Augmentations before and after base points and type II
  moves} \label{sec:basePointChanges}

In some cases, we will find that adding base points will simplify the
signs. For Legendrian links in $\reals^3$, Ng and Rutherford give the
DGA isomorphisms induced by adding a base point to a diagram and by
moving a base point around a link in \cite{NgSatellites}. One can
easily extend their results to $\#^k(S^1\times S^2)$. In the case
where a base point is pushed through a crossing $c_i$, the DGA
isomorphism sends $c_i$ to $t_j^{\pm1}c_i$, the sign depending on
whether the base point is pushed along the link with or against the
orientation of the strand, and preserves $c_j$ if no base point is
pushed through $c_j$. If a base point $*_i$ corresponding to $t_i$ is
added next to a base point $*$ corresponding to $t$, then the DGA
homomorphism sends $t$ to $tt_i^{-1}$. Given an augmentation of the DGA of
the diagram before either operation, this DGA isomorphism clearly
gives us an augmentation of the DGA of the new diagram.

\begin{rmk}
  In summary, if we have an augmentation $\epsilon:\A\to F$ with
  $\epsilon(t_i)=-1$, then moving the base point $*_i$ through a
  crossing $c_j$ only changes the augmentation by changing the sign of
  the augmentation on the crossing $c_j$. Suppose we have a diagram
  with a base point $*$ corresponding to $t$ and the same diagram with
  base points $*_1,\ldots,*_s$ associated to $t_1,\ldots, t_s$ on the
  same component of the link and we move all of the base points
  $*_1,\ldots,*_s$ to the location of $*$. By the above results, if
  $\epsilon$ is an augmentation to $F$ of the multiple base point
  diagram, there exists an augmentation $\epsilon'$ to $F$ of the
  single base point diagram such that for all crossings $c$ there
  exists $x_c\in F$ such that $\epsilon'(c)=x_c\epsilon(c)$ and
  \[\epsilon'(t)=\epsilon(t_1\ldots t_s)=\prod_{i=1}^s\epsilon(t_i).\]
\end{rmk}

In \cite{EtnyreInvariants}, Etnyre, Ng, and Sabloff give a DGA
isomorphism relating the DGA of a diagram of a Legendrian knot in
$\reals^3$ before and after a Reidemeister II move. One can easily
extend this to a similar result for $\#^k(S^1\times S^2)$, which gives
a way to extend an augmentation of the diagram before a Reidemeister
II move to an augmentation of the diagram after the move, (see
\cite{Leverson} for the analogous result in $\reals^3$).

\section{Correspondence between augmentations and normal rulings for
  links in $\reals^3$}\label{sec:correspLinks}

From \cite{Leverson}, we have the following result for \emph{knots} in
$\reals^3$.

\begin{thm}[\cite{Leverson} Theorem 1.1]
  Let $\Lambda$ be a Legendrian knot in $\reals^3$. Given a field $F$,
  $(\A,\partial)$ has a $\rho$-graded augmentation $\epsilon:\A\to F$
  if and only if any front diagram of $\Lambda$ has a $\rho$-graded
  normal ruling. Furthermore, if $\rho$ is even, then
  $\epsilon(t)=-1$.
\end{thm}

This result is proven by construction. Using the same method we can
prove an analogous result for \emph{links} in $\reals^3$. Restating
from the introduction:

\begin{thmLinksInR}
  Let $\Lambda$ be an $n$-component Legendrian link in $\reals^3$ with
  $s$ base points (at least one base point on each component). Given a
  field $F$, the Chekanov-Eliashberg DGA $(\A,\partial)$ over
  $\integers[t_1^{\pm1},\ldots,t_s^{\pm1}]$ has a $\rho$-graded
  augmentation $\epsilon:\A\to F$ if and only if a front diagram of
  $\Lambda$ has a $\rho$-graded normal ruling. Furthermore, if $\rho$
  is even, then $\epsilon(t_1\cdots t_s)=(-1)^s$.
\end{thmLinksInR}

The following result will be necessary for the proof of
Theorem~\ref{thm:linksInR}. Analogous to the knot case in $\reals^3$,
we have the following extension of Lemma~3.2 from (\cite{Leverson}):
\begin{lem}\label{lem:oddNumBasepts}
  If $c$ gives the number of right cusps, $sw$ is the number of
  switches in the ruling, $a_-$ is the number of $-$(a) crossings, and
  $n$ the number of components then
  \[c+sw+a_-\equiv n\mod2.\]
\end{lem}

\begin{proof}
  As in the knot case, one can easily show each of the following
  statements:
  \begin{align}
    &\sum_{i=1}^ntb_i+\sum_{i=1}^nr_i\equiv n\mod2\label{state1}\\
    &\sum_{i=1}^ntb_i\equiv c+cr\mod2\label{state2}\\
    &cr\equiv sw\mod2\label{state3}\\
    &\sum_{i=1}^nr_i\equiv a_-\mod2\label{state4}
  \end{align}
  where $r_i$ is the rotation number of $\Lambda_i$ and $cr$ is the
  number of crossings. Note that if we add these four equations
  together, we get that
  \[c+sw+a_-\equiv n\mod2\] as desired.
\end{proof}

\begin{proof}[Proof of Theorem~\ref{thm:linksInR}]
  After a series of Legendrian isotopies, we can assume the front
  diagram of $\Lambda$ has the following form where from left to right
  (lowest $x$-coordinate to highest $x$-coordinate) we have: all left
  cusps have the same $x$-coordinate, no two crossings of $\Lambda$
  have the same $x$-coordinate, and all right cusps have the same
  $x$-coordinate (in \cite{Leverson}, this is called plat
  position). Label the crossings in the right cusps by
  $q_1,\ldots,q_m$ from top to bottom and label the other crossings
  by $c_1,\ldots,c_\ell$ from left to right.

  (Augmentation to ruling) Given a $\rho$-graded augmentation of the
  Chekanov-Eliashberg DGA of the resolution of $\pi_{xz}(\Lambda)$ to
  a Lagrangian diagram. Define a $\rho$-graded normal ruling of
  $\pi_{xz}(\Lambda)$ by simultaneously defining a $\rho$-graded
  augmentation of the dipped diagram $D(\Lambda)$ as in the knot case,
  using Figure~\ref{fig:basepoints}.

  (Ruling to augmentation) Given a $\rho$-graded normal ruling of
  $\pi_{xz}(\Lambda)$. Define a $\rho$-graded augmentation of the
  dipped diagram $D(\Lambda)$ with base points where specified in
  Figure~\ref{fig:basepoints} and at each right cusps as in the knot
  case, using Figure~\ref{fig:basepoints}.

  Using Lemma~\ref{lem:oddNumBasepts} and the methods in the proof of
  Theorem~3.1 in \cite{Leverson}, one can show the final statement of
  Theorem~\ref{thm:linksInR}. Given a $\rho$-graded augmentation
  $\epsilon:\A\to F$, consider the associated $\rho$-graded normal
  ruling. If $\rho$ is even, then the ruling is only switched at
  crossings $c_k$ with $\rho\big\vert\lvert c_k\rvert$ and so
  $2\big\vert\lvert c_k\rvert$. Thus, any strands paired by the ruling
  must have opposite orientation. As in the case of knots, this
  implies that near a crossing where the ruling is switched the
  crossing must be a positive crossing. Thus each ruling path is an
  oriented unknot.

  If we consider the dipped diagram of the link, by induction we can
  show that
  \[\prod\left(\epsilon(b^k_{ij})^{\pm1}\right)=1,\]
  where the product is taken over all paired strands $i$ and $j$ in
  the ruling between $c_k$ and $c_{k+1}$ and the sign is determined by
  the orientation of the paired strands as in \cite{Leverson}. By
  considering $\partial q_k$, we see that
  \begin{align*}
    \epsilon(t_1\cdots t_s)&=(-1)^{s-m}\prod_{k=1}^m\left(-(\epsilon(b^\ell_{2k,2k-1}))^{\pm1}\right)\\
    &=(-1)^s\prod_{i<j\text{ paired}}\left(\epsilon(b^\ell_{ij})^{\pm1}\right)\\
    &=(-1)^s\\
    &=(-1)^n
  \end{align*}
  by Lemma~\ref{lem:oddNumBasepts} and the fact that the number of
  base points $s\equiv c+sw+a_-\mod2$.
\end{proof}

\section{Augmentation to Ruling}\label{sec:augToRuling}

In this section, we will show that the DGA of a Legendrian link
$\Lambda$ in $\#^k(S^1\times S^2)$ is a subalgebra of the DGA of
satellited $\Lambda$ in $\reals^3$ and use the construction from
Theorem 1.1 \cite{Leverson} to construct a ruling of the satellited
link in $\reals^3$ to then give a normal ruling of $\Lambda$ in
$\#^k(S^1\times S^2)$. This shows the forward direction of Theorem
\ref{thm:main}.

Given an $xy$-diagram for the Legendrian link $\Lambda$ in
$\#^k(S^1\times S^2)$ which results from the resolution of an
$xz$-diagram in normal form with base points indicated. We can
construct an $xy$-diagram for $S(\Lambda)$, satellited $\Lambda$, (see
Figure~\ref{fig:exSatellitedxy}) with base points in the same location
as they were for $\Lambda$.

We will use the notation for Legendrian links in $\#^k(S^1\times S^2)$
with tildes added for the Legendrian link $\Lambda$ in $\#^k(S^1\times
S^2)$:
$\A(\Lambda)=\integers[\tilde{t}_1^{\pm1},\ldots,\tilde{t}_s^{\pm1}]\langle
\tilde{a}_i,\tilde{b}_{ij;\ell},\tilde{c}^p_{ij;\ell}\rangle$ with
differential $\tpartial$, where $1\leq\ell\leq k$, $i<j$ for all
$\tb_{ij;\ell}$, $i<j$ for $\tc^p_{ij;\ell}$ if $p=1$, and $i\leq j$
if $p>1$. We will use the notation for Legendrian links from
Figure~\ref{fig:exSatellitedxy} for $S(\Lambda)$:
\[\A(S(\Lambda))=\integers[t_1^{\pm1},\ldots,t_s^{\pm1}]\langle
a_i,b_{ij;\ell},c_{ij;\ell},d_{ji;\ell},e_{ij;\ell},f_{ji;\ell},g_{ij;\ell},h_{ji;\ell},q_{ij;\ell}\rangle\]
with differential $\partial$, where $1\leq\ell\leq k$, $1\leq i\leq m$
for $a_i$, $i<j$ for $b_{ij;\ell}$, $c_{ij;\ell}$, $e_{ij;\ell}$,
$g_{ij;\ell}$, and $q_{ij;\ell}$, and $i\leq j$ for $d_{ji;\ell}$,
$f_{ji;\ell}$, and $h_{ji;\ell}$.

Note that
\begin{align*}
  \partial
  a_i&=\tilde\partial\tilde{a}_i\rvert_{\tilde{a}_r=a_r,\tilde{c}^0_{rs;p}=q_{rs;p},\tilde{t}_r=t_r},\\
  \partial
  b_{ij;\ell}&=\tilde\partial\tilde{b}_{ij;\ell}\rvert_{\tilde{a}_r=a_r,\tilde{b}_{rs;p}=b_{rs;p},\tilde{c}^0_{rs;p}=q_{rs;p},\tilde{t}_r=t_r},
\end{align*}
and in the $p$-th $1$-handle
\begin{align*}
  \partial
  c_{ij}&=\tilde\partial\tilde{c}^0_{ij}\rvert_{\tilde{c}^0_{rs}=c_{rs}},
\end{align*}
where $1\leq i<j\leq N_p$. One can check that in the $p$-th
$1$-handle
\[\partial e_{ij}=\sum_{i<\ell<j}(-1)^{\lvert
  e_{i\ell}\rvert+1}e_{i\ell}e_{\ell
  j}=\tilde\partial\tilde{c}^0_{ij}\rvert_{\tilde{c}^0_{rs}=e_{rs}}\]
for $1\leq i<j\leq N_p$. Similarly
\begin{align*}
  \partial g_{ij}&=\tilde\partial\tilde{c}^0_{ij}\rvert_{\tilde{c}^0_{rs}=g_{rs}},\\
  \partial
  q_{ij}&=\tilde\partial\tilde{c}^0_{ij}\rvert_{\tilde{c}^0_{rs}=q_{rs}}.
\end{align*}
One can also check that
\begin{align*}
  \partial d_{ji}&=\delta_{ij}+\sum_{j<\ell\leq N_p}c_{j\ell}d_{\ell i}+\sum_{1\leq\ell<i}(-1)^{\lvert d_{j\ell}\rvert+1}d_{j\ell}e_{\ell i},\\
  \partial f_{ji}&=\delta_{ij}+\sum_{j<\ell\leq N_p}(-1)^{\lvert e_{j\ell}\rvert+1}e_{j\ell}f_{\ell i}+\sum_{1\leq\ell<i}(-1)^{\lvert f_{j\ell}\rvert+\lvert g_{\ell i}\rvert}f_{j\ell}g_{\ell i},\\
  \partial h_{ji}&=\delta_{ij}+\sum_{j<\ell\leq N_p}(-1)^{\lvert
    q_{j\ell}\rvert+1}q_{j\ell}h_{\ell
    i}+\sum_{1\leq\ell<i}(-1)^{\lvert
    h_{j\ell}\rvert+1}h_{j\ell}g_{\ell i},
\end{align*}
where $1\leq i\leq j\leq N_p$.

\begin{rmk}
  Suppose we have a Legendrian link $\Lambda$ in $\#^k(S^1\times S^2)$
  with associated DGA $(\A(\Lambda),\partial)$. If
  $(\A(S(\Lambda)),\partial)$ is the DGA associated to satellited
  $\Lambda$, then we have
  \diag{\A(S(\Lambda))\arr&\A(S(\Lambda))/B\,\ar@{^{(}->}[r]&\A(\Lambda),}
  where the final map is inclusion and
  \[B=R\langle
  c_{ij;\ell}-g_{ij;\ell},c_{ij;\ell}-q_{ij;\ell},c_{ij;\ell}-(-1)^{\lvert
    e_{ij;\ell}\rvert+1}e_{ij;\ell},h_{ji;\ell}-(-1)^{\lvert
    f_{ji;\ell}\rvert+1}f_{ji;\ell},h_{ji;\ell}-(-1)^{\lvert
    d_{ji;\ell}\rvert+1}d_{ji;\ell}\rangle.\]
\end{rmk}

Given a field $F$ and a $\rho$-graded augmentation
$\tilde\epsilon:\A(\Lambda)\to F$ we will construct a $\rho$-graded
augmentation $\epsilon:\A(S(\Lambda))\to F$. Define $\epsilon$ on the
generators of $\A(S(\Lambda))$ by
\[\epsilon(c)=\begin{cases}
  \tepsilon(\tilde{a}_i)&\text{if }c=a_i\\
  \tepsilon(\tilde{b}_{ij})&\text{if }c=b_{ij}\\
  \tepsilon(\tilde{c}^0_{ij})&\text{if }c\in\{c_{ij},g_{ij},q_{ij}\}\\
  (-1)^{\lvert\tilde{c}^0_{ij}\rvert+1}\tepsilon(\tilde{c}^0_{ij})&\text{if }c=e_{ij}\\
  \tepsilon(\tilde{c}^1_{ji})&\text{if }c=h_{ji}\\
  (-1)^{\lvert\tilde{c}^1_{ji}\rvert+1}\tepsilon(\tilde{c}^1_{ji})&\text{if }c\in\{d_{ji},f_{ji}\}\\
  \tepsilon(\tilde{t}_i)&\text{if }c=t_i
\end{cases}\] in the $\ell$-th $1$-handle.

\begin{rmk}\label{rmk:grading}
  Note that for fixed $i,j,$ and $p$,
  $\tilde{c}^0_{ij;p},c_{ij;p},d_{ji;p},e_{ij;p},f_{ji;p},g_{ij;p},h_{ji;p},$
  and $q_{ij;p}$ are either all positive crossings or all negative
  crossings. We also note that for a given $1$-handle,
  $\lvert\tilde{c}^0_{ij}\rvert\equiv\lvert\tilde{c}^1_{ji}\rvert\mod
  2$ and
  $\lvert\tilde{c}^1_{ij}\rvert\equiv\lvert\tilde{c}^1_{ji}\rvert\mod2$. Therefore,
  for a given $1$-handle, the following are all congruent mod $2$:
  \[\lvert\tilde{c}^0_{ij}\rvert,\lvert\tilde{c}^1_{ij}\rvert,\lvert\tilde{c}^1_{ji}\rvert,\lvert
  c_{ij}\rvert,\lvert d_{ji}\rvert,\lvert e_{ij}\rvert,\lvert
  f_{ji}\rvert,\lvert g_{ij}\rvert,\lvert h_{ji}\rvert,\lvert
  q_{ij}\rvert.\]
\end{rmk}

We will now check that $\epsilon$ is a $\rho$-graded augmentation of
$(\A(S(\Lambda)),\partial)$. Clearly in the $p$-th $1$-handle
\[\epsilon\partial a_r=\epsilon\partial
b_{ij}=\epsilon\partial c_{ij}=\epsilon\partial
g_{ij}=\epsilon\partial q_{ij}=0\] for all $1\leq r\leq m$ and $1\leq
i<j\leq N_p$. Note that in the $p$-th $1$-handle
\begin{align}\label{eq:gradingSum}
  \lvert\tilde{c}^0_{ij}\vert&\equiv\lvert\tilde{c}^0_{i\ell}\rvert+\lvert\tilde{c}^0_{\ell
    j}\rvert\mod2\\
  \lvert\tilde{c}^1_{ji}\rvert&\equiv\lvert\tilde{c}^1_{j\ell}\rvert+\lvert\tilde{c}^1_{\ell
    i}\rvert\mod2\nonumber
\end{align}
Given $1\leq p\leq k$ and $1\leq i<j\leq N_p$. In the $p$-th
$1$-handle:
\begin{align*}
  \epsilon\partial e_{ij}&=\sum_{i<\ell<j}(-1)^{\lvert e_{i\ell}\rvert+1}\epsilon(e_{i\ell}e_{\ell j})\\
  &=\sum_{i<\ell<j}(-1)^{\lvert e_{\ell j}\rvert+1}\tepsilon(\tilde{c}^0_{i\ell}\tilde{c}^0_{\ell j})\\
  &=\sum_{i<\ell<j}(-1)^{\lvert\tilde{c}^0_{ij}\rvert+\lvert\tilde{c}^0_{i\ell}\rvert+1}\tepsilon(\tilde{c}^0_{i\ell}\tilde{c}^0_{\ell j})\text{ by \eqref{eq:gradingSum}}\\
  &=(-1)^{\lvert\tilde{c}^0_{ij}\rvert+1}\tepsilon\partial\tilde{c}^0_{ij}\\
  &=0;\displaybreak[2]\\[.2in]
  \epsilon\partial d_{ji}&=\sum_{j<\ell\leq N_p}\epsilon(c_{j\ell}d_{\ell i})+\sum_{1\leq\ell<i}(-1)^{\lvert d_{j\ell}\rvert+1}\epsilon(d_{j\ell}e_{\ell i})\\
  &=\sum_{j<\ell\leq N_p}(-1)^{\lvert\tilde{c}^1_{\ell i}\rvert+1}\tepsilon(\tilde{c}^0_{j\ell}\tilde{c}^1_{\ell i})+\sum_{1\leq\ell<i}(-1)^{\lvert\tilde{c}^1_{\ell i}\rvert+1}\tepsilon(\tilde{c}^1_{j\ell}\tilde{c}^0_{\ell i})\text{ by Remark \ref{rmk:grading}}\\
  &=\sum_{j<\ell\leq N_p}(-1)^{\lvert\tilde{c}^1_{ji}\rvert+\lvert\tilde{c}^1_{j\ell}\rvert+1}\tepsilon(\tilde{c}^0_{j\ell}\tilde{c}^1_{\ell i})+\sum_{1\leq\ell<i}(-1)^{\lvert\tilde{c}^1_{ji}\rvert+\lvert\tilde{c}^1_{j\ell}\rvert+1}\tepsilon(\tilde{c}^1_{j\ell}\tilde{c}^0_{\ell i}) \text{ by \eqref{eq:gradingSum}}\\
  &=(-1)^{\lvert\tc^1_{ji}\rvert}\tepsilon\tpartial\tc^1_{ji} \text{ by Remark~\ref{rmk:grading}}\\
  &=0\displaybreak[2]\\[.2in]
  \epsilon\partial d_{jj}&=1+\sum_{j<\ell\leq N_p}\epsilon(c_{j\ell}d_{\ell j})+\sum_{1\leq\ell<j}(-1)^{\lvert d_{j\ell}\rvert+1}\epsilon(d_{j\ell}e_{\ell j})\\
  &=1+\sum_{j<\ell\leq N_p}(-1)^{\lvert\tc^0_{j\ell}\rvert+1}\tepsilon(\tc^0_{j\ell}\tc^p_{1\ell j})+\sum_{1\leq\ell<j}(-1)^{\vert\tc^1_{j\ell}\rvert+1}\tepsilon(\tc^1_{j\ell}\tc^0_{\ell j})\text{ by Remark~\ref{rmk:grading}}\\
  &=\tepsilon\tpartial\tc^1_{jj}\\
  &=0;\displaybreak[2]\\[.2in]
\end{align*}

Similarly one can show $\epsilon\partial f_{ij}=0$ if $i\leq j$ and
$\epsilon\partial h_{ji}=0$ if $i<j$.

(grading) If $\tepsilon$ is $\rho$-graded, we will show that
$\epsilon$ is $\rho$-graded as well. Let $m$ be the Maslov potential
used to assign the gradings of the crossings of $\Lambda$ in
$\#^k(S^1\times S^2)$. We will use $m$ to define a Maslov potential
$\mu$ on $S(\Lambda)$ in $\reals^3$ as follows: Define $\mu$ on
$T\subset S(\Lambda)$ the same as $m$ is defined on $T\subset\Lambda$
and extend $\mu$ to the rest of $S(\Lambda)$. Notice that there is
only one way to do this which keeps $\mu$ of the upper strand (higher
$z$-coordinate) entering a cusp one higher than $\mu$ of the lower
strand (lower $z$-coordinate) entering a cusp. Thus it is clear that
$\lvert \ta_i\vert=\lvert a_i\rvert$,
$\lvert\tb_{ij;\ell}\rvert=\lvert b_{ij;\ell}\rvert$, and
$\lvert\tc_{ij;\ell}\rvert=\lvert c^0_{ij;\ell}\rvert$. Properties of
the Maslov potential immediately give us
\begin{align*}
  \lvert d_{ji}\rvert=\lvert f_{ji}\rvert=\lvert h_{ji}\rvert,\quad i\leq j\\
  \lvert e_{ij}\rvert=\lvert g_{ij}\rvert=\lvert q_{ij}\rvert,\quad i<j\\
  -\lvert d_{ji}\rvert=\lvert e_{ij}\rvert,\quad i<j
\end{align*}
Therefore, it suffices to check that $\rho\big\vert\lvert\tc^0_{ij}\rvert$
if and only if $\rho\big\vert\lvert d_{ji}\rvert$ for $i<j.$

To this end, we note that
$\lvert\tb_{ij}\rvert=\lvert\tc^0_{ij}\rvert+1$ and
$\lvert\tb_{ij}\rvert=m(i)-m(j)$, so
$\lvert\tc^0_{ij}\rvert=m(i)-m(j)-1$. Thus, by the definition of
$\mu$, we have
\[\lvert d_{ji}\rvert=m(j)-(m(i)-1)=-\lvert\tc^0_{ij}\rvert.\]
So $\epsilon$ is $\rho$-graded if $\tepsilon$ is $\rho$-graded.

Thus an augmentation $\tepsilon:\A(\Lambda)\to F$ of the DGA of
$\Lambda$ in $\#^k(S^1\times S^2)$ gives an augmentation
$\epsilon:\A(S(\Lambda))\to F$ of the DGA of $S(\Lambda)$ in
$\reals^3$. By Theorem 1.1 in \cite{Leverson}, the augmentation
$\epsilon$ gives an augmentation of the DGA of $S(\Lambda)$ with dips
in $\reals^3$, which gives a normal ruling of $S(\Lambda)$ with no
dips in $\reals^3$. Clearly this normal ruling must be {\bf thin},
meaning outside of the tangle $T$ associated to $\Lambda$ the ruling
only has switches at crossings where the crossing strands go through
the same $1$-handle. By restricting the $\rho$-graded normal ruling of
$S(\Lambda)$ in $\reals^3$ to a $\rho$-graded normal ruling of $T$, we
get a $\rho$-graded normal ruling of $\Lambda$ in $\#^k(S^1\times
S^2)$.

An easy to prove corollary of this is:

\begin{cor}\label{cor:oddNumStrands}
  If $\Lambda$ is a Legendrian link in $\#^k(S^1\times S^2)$ and there
  exists $\ell$ such that $N_\ell$ is odd, then there does
  not exist a $\rho$-graded augmentation of the DGA $\A(\Lambda)$ for
  any $\rho$.
\end{cor}

In other words, if $\Lambda$ has a $1$-handle with an odd number of
strands going through it, then there does not exist a $\rho$-graded
augmentation of the DGA $\A(\Lambda)$ for any $\rho$.

\begin{proof}
  It is clear that any normal ruling of $S(\Lambda)$ must be thin, but
  if $\Lambda$ has a $1$-handle with an odd number of strands going
  through it, then there are no thin normal rulings of $S(\Lambda)$
  and thus no normal rulings of $S(\Lambda)$. So
  Theorem~\ref{thm:main} tells us there are no $\rho$-graded
  augmentations of $\A(\Lambda)$.
\end{proof}

\section{Ruling to Augmentation}\label{sec:rulingToAug}

Let $F$ be a field. We will now prove the existence of a $\rho$-graded
normal ruling implies the existence of a $\rho$-graded augmentation,
the backward direction of Theorem~\ref{thm:main}, by constructing a
$\rho$-graded augmentation $\epsilon:\A(D(\Lambda))\to F$ given a
$\rho$-graded normal ruling of $\Lambda$ in $\#^k(S^1\times S^2)$.

Given an $xz$-diagram of a Legendrian link $\Lambda$ in
$\#^k(S^1\times S^2)$ in normal form, we will consider the resolution
to an $xy$-diagram of a Legendrian isotopic link. Using Legendrian
isotopy, we can ensure all crossings, left cusps, and right cusps have
different $x$ coordinates and all right cusps occur ``above'' (have
higher $y$ or $z$ coordinate than) the remaining strands fo the tangle
at that $x$ coordinate. Place a base point on every strand at $x=0$
and one in every loop coming from the resolution of a right cusp.

\begin{figure}
  \labellist
  \small\hair 3pt 
  \pinlabel {$-$(a)} [b] at 272 1863 
  \pinlabel $a_1$ [tl] at 147 1687 
  \pinlabel $a_2$ [tl] at 212 1750 
  \pinlabel $a$ [b] at 272 1810 
  \pinlabel $a^{-1}$ [tr] at 360 1723 
  \pinlabel $aa_1$ [tl] at 440 1687 
  \pinlabel $aa_2$ [tl] at 505 1750
  
  \pinlabel {$+$(a)} [b] at 972 1863 
  \pinlabel $a_1$ [tl] at 847 1687 
  \pinlabel $a_2$ [tl] at 912 1750 
  \pinlabel $a$ [b] at 972 1810 
  \pinlabel $a^{-1}$ [tr] at 1064 1723 
  \pinlabel $aa_1$ [tl] at 1146 1687 
  \pinlabel $aa_2$ [tl] at 1207 1750
   
  \pinlabel {$-$(b)} [b] at 277 1535 
  \pinlabel $a_1$ [tr] at 149 1424 
  \pinlabel $a_2$ [tl] at 182 1392 
  \pinlabel $a$ [b] at 275 1510 
  \pinlabel $a^{-1}a_1a_2^{-1}$ [tr] at 337 1435 
  \pinlabel $a^{-1}$ [tr] at 392 1359 
  \pinlabel $a^{-1}a_1$ [tl] at 436 1438
  \pinlabel $aa_2$ [tl] at 472 1389
   
  \pinlabel {$+$(b)} [b] at 972 1535 
  \pinlabel $a_1$ [tr] at 859 1421 
  \pinlabel $a_2$ [tl] at 881 1392 
  \pinlabel $a$ [b] at 975 1510 
  \pinlabel $a^{-1}a_1a_2^{-1}$ [tr] at 1037 1435 
  \pinlabel $a^{-1}$ [tr] at 1100 1362 
  \pinlabel $a^{-1}a_1$ [tl] at 1138 1438
  \pinlabel $aa_2$ [tl] at 1174 1389
  
  \pinlabel {$-$(c), product of signs of $a^{j-1}_{Li}$ and $a^{j-1}_{i+1,K}$ is $+1$} [b] at 272 1232 
  \pinlabel {$+$(c), product of signs of $a^{j-1}_{Li}$ and $a^{j-1}_{i+1,K}$ is
    $-1$} [b] at 272 1192 
  \pinlabel $a_1$ [tr] at 155 1084 
  \pinlabel $a_2$ [tl] at 181 1051
  \pinlabel $a$ [b] at 273 1105
  \pinlabel $a^{-1}a_1a_2^{-1}$ [tr] at 392 1025 
  \pinlabel $a^{-1}$ [tr] at 336 1084 
  \pinlabel $a^{-1}a_1$ [tl] at 436 1098
  \pinlabel $aa_2$ [tl] at 472 1051
   
  \pinlabel {$-$(c), product of signs of $a^{j-1}_{Li}$ and
    $a^{j-1}_{i+1,K}$ is $-1$} [b] at 972 1232
  \pinlabel {$+$(c), product of signs of $a^{j-1}_{Li}$ and $a^{j-1}_{i+1,K}$ is
    $+1$} [b] at 972 1192 
  \pinlabel $a_1$ [tr] at 860 1084 
  \pinlabel $a_2$ [tl] at 881 1051 
  \pinlabel $a$ [b] at 975 1105 
  \pinlabel $a^{-1}a_1a_2^{-1}$ [tr] at 1103 1025 
  \pinlabel $a^{-1}$ [tr] at 1036 1084
  \pinlabel $a^{-1}a_1$ [tl] at 1142 1098
  \pinlabel $aa_2$ [tl] at 1174 1051

  \pinlabel {(d)} [b] at 271 870 
  \pinlabel $a_1$ [tl] at 145 729
  \pinlabel $a_2$ [tl] at 179 760
  \pinlabel $a$ [b] at 273 812
  \pinlabel $a_1$ [tl] at 440 696
  \pinlabel $a_2$ [tl] at 505 760

  \pinlabel {$-$(e)} [b] at 271 543
  \pinlabel $a_1$ [tl] at 145 398
  \pinlabel $a_2$ [tl] at 177 431
  \pinlabel $a$ [b] at 275 514
  \pinlabel $aa_1^{-1}a_2$ [tr] at 360 438
  \pinlabel $a_2$ [tl] at 435 431 
  \pinlabel $a_1$ [tl] at 472 398
   
  \pinlabel {$+$(e)} [b] at 971 543 
  \pinlabel $a_1$ [tl] at 847 398
  \pinlabel $a_2$ [tl] at 879 431
  \pinlabel $a$ [b] at 975 514
  \pinlabel $aa_1^{-1}a_2$ [tr] at 1059 438
  \pinlabel $a_2$ [tl] at 1139 430
  \pinlabel $a_1$ [tl] at 1172 398

  \pinlabel {$-$(f), product of signs of $a^{j-1}_{L,i+1}$ and
    $a^{j-1}_{iK}$ is $+1$} [b] at 225 247
  \pinlabel {$+$(f), product of signs of $a^{j-1}_{L,i+1}$ and $a^{j-1}_{iK}$ is
    $-1$} [b] at 225 207
  \pinlabel $a_1$ [tl] at 147 67
  \pinlabel $a_2$ [tl] at 180 99
  \pinlabel $a$ [b] at 273 118
  \pinlabel $aa_1a_2^{-1}$ [tr] at 394 35
  \pinlabel $a_1$ [tl] at 436 99
  \pinlabel $a_2$ [tl] at 472 67
   
  \pinlabel {$-$(f), product of signs of $a^{j-1}_{L,i+1}$ and
    $a^{j-1}_{iK}$ is $-1$} [b] at 1018 247
  \pinlabel {$+$(f), product of signs of $a^{j-1}_{L,i+1}$ and $a^{j-1}_{iK}$ is
    $+1$} [b] at 1018 207 
  \pinlabel $a_1$ [tl] at 847 75 
  \pinlabel $a_2$ [tl] at 880 110
  \pinlabel $a$ [b] at 980 125
  \pinlabel $aa_1a_2^{-1}$ [tr] at 1105 45
  \pinlabel $a_1$ [tl] at 1140 110
  \pinlabel $a_2$ [tl] at 1173 72

  \endlabellist
  \includegraphics[width=.9\textwidth]{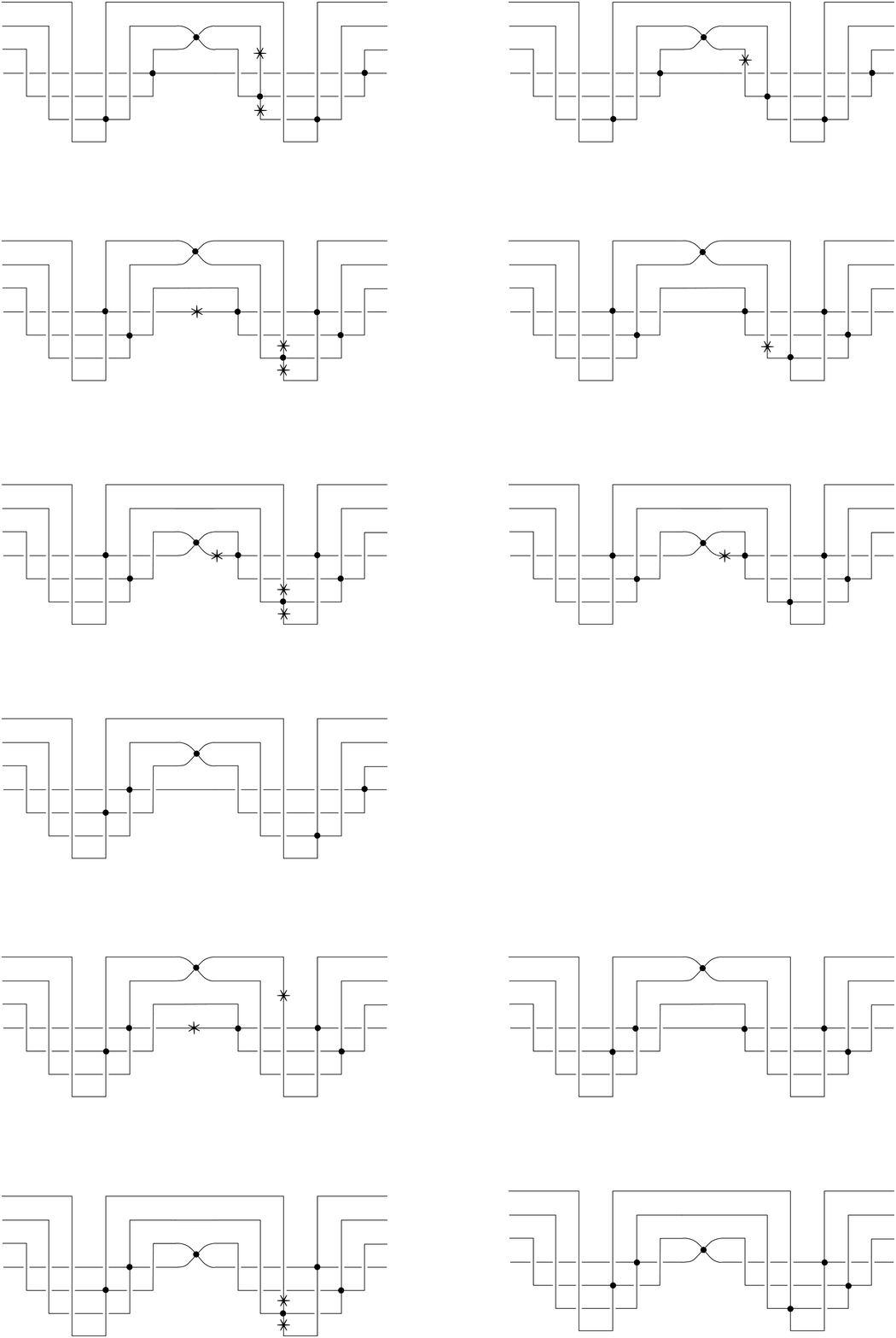}
  \caption{In the diagrams, $*$ denotes a base point. A dot denotes
    the specified crossing is augmented and the augmentation sends the
    crossing to the label. Here $-/+$(a) denotes a negative/positive
    crossing where the ruling has configuration (a) and the rest are
    defined analogously. (This figure is taken from \cite{Leverson}.)}
  \label{fig:basepoints}
\end{figure}

Define the augmentation $\epsilon:\A(D(\Lambda))\to F$ of the DGA for
the dipped diagram $D(\Lambda)$ on generators as follows: If the
ruling is switched at a crossing $a_\ell$, then set
$\epsilon(a_\ell)=1$. If not, set $\epsilon(a_\ell)=0$. (Note that we
can augment the switched crossings to any nonzero element of $F$ and
still get an augmentation. But in the case where $\Lambda$ is a knot,
by augmenting the switched crossing to $1$, we will be able to ensure
$\epsilon(t)=-1$.) Add base points and augment the crossings in the
dips following Figure~\ref{fig:basepoints}. On the remaining
generators, set
\[\epsilon(c^\ell_{ij})=\begin{cases}
  1&\text{if }\ell=0\text{ and strands }i,j\text{ are paired in the normal ruling and go through the }p\text{-th }1\text{-handle}\\
  (-1)^{\lvert c^\ell_{ij}\rvert}&\text{if }\ell=1, i>j,\text{ and strands }i,j\text{ are paired in the normal ruling and go through the }p\text{-th }1\text{-handle}\\
  0&\other.
\end{cases}\] Augment all base points to $-1$.

By considering Figure~\ref{fig:basepoints}, one can check that
$\epsilon$ is an augmentation on the $a_\ell$ and the crossings in the
dips.

\begin{note}
  $c^\ell_{\{ij\}}=c^\ell_{\min(i,j),\max(i,j)}$
\end{note}

We will now check that $\epsilon$ is an augmentation on the
$c^\ell_{ij}$ generators from the $p$-th $1$-handle.

($\epsilon\partial c^0_{ij}=0$) For any ruling, at the left end of
the diagram, each strand is paired with another strand going through
the same $1$-handle. So for each strand $i$ going through the $p$-th
$1$-handle, there exists a strand $j\neq i$ such that strand $i$ and
$j$ are paired and $1\leq i,j\leq N_p$. So if $i<j$, then
$\epsilon(c^0_{ij})=1$, $\epsilon(c^0_{\{i\ell\}})=0$ for all
$\ell\neq j$, and $\epsilon(c^0_{\{j\ell\}})=0$ for all $\ell\neq
i$. Suppose $i<r<\ell$. We see that $\epsilon(c^0_{ir})=0$ if $r\neq
j$ and $\epsilon(c^0_{r\ell})=0$ if $r=j$. Thus
$\epsilon(c^0_{ir}c^0_{r\ell})=0$ for all $i<r<\ell$ and so
\[\epsilon\partial c^0_{i\ell}=\sum_{i<r<\ell}(-1)^{\lvert
  c^0_{ir}\rvert+1}\epsilon(c^0_{ir}c^0_{r\ell})=0\] for $i<\ell$.

($\epsilon\partial c^1_{ij}=0$) Recall that in the $p$-th $1$-handle
\[\partial c^1_{ij}=\delta_{ij}+\sum_{i<\ell\leq N_p}(-1)^{\lvert
  c^0_{i\ell}\rvert+1}c^0_{i\ell}c^1_{\ell
  j}+\sum_{1\leq\ell<j}(-1)^{\lvert
  c^1_{i\ell}\rvert+1}c^1_{i\ell}c^0_{\ell j}.\] If $i\neq j$, then
$\epsilon(c^0_{i\ell}c^1_{\ell j})=0$ and
$\epsilon(c^1_{i\ell}c^0_{\ell j})=0$ for all $\ell$ since it is not
possible for strand $i$ to be paired with strand $\ell$ and for strand
$\ell$ to be paired with strand $j$ when $i\neq j$. Thus
\[\epsilon\partial c^1_{ij}=\sum_{i<\ell\leq
  N_p}(-1)^{\lvert
  c^0_{i\ell}\rvert+1}\epsilon(c^0_{i\ell}c^1_{\ell
  j})+\sum_{1\leq\ell<j}(-1)^{\lvert
  c^1_{i\ell}\rvert+1}\epsilon(c^1_{i\ell}c^0_{\ell j})=0.\]

To show $\epsilon\partial c^1_{ii}=0$, suppose strand $i$ is paired
with strand $\ell$ through the $p$-th $1$-handle. Then
\begin{align*}
  \epsilon\partial c^1_{ii}&=\begin{cases}
    1+(-1)^{\lvert c^0_{i\ell}\rvert+1}\epsilon(c^0_{i\ell}c^1_{\ell i})&i<j\\
    1+(-1)^{\lvert c^1_{i\ell}\rvert+1}\epsilon(c^1_{i\ell}c^0_{\ell i})&i>j
  \end{cases}\\
  &=\begin{cases}
    1+(-1)^{\lvert c^0_{i\ell}\rvert+1}(-1)^{\lvert c^1_{\ell i}\rvert}&i<j\\
    1+(-1)^{\lvert c^1_{i\ell}\rvert+1}(-1)^{\lvert
      c^1_{i\ell}\rvert}&i>j
  \end{cases}\\
  &=0
\end{align*}
by Remark~\ref{rmk:grading}.

($\epsilon\partial c^\ell_{ij}=0$ for $1<\ell$) Recall
\[\partial c^\ell_{ij}=\sum_{r=0}^\ell\sum_{s=1}^{N_p}(-1)^{\lvert
  c^r_{is}\rvert+1}c^r_{is}c^{\ell-r}_{sj}\] for $1<\ell$, $1\leq
p\leq k$, and $1\leq i,j\leq N_p$. We will show that
\[\epsilon(c^r_{is}c^{\ell-r}_{sj})=0,\]
which implies that $\epsilon\partial c^\ell_{ij}=0$.  If $\ell>2$,
then for all $0\leq r\leq\ell$, either $r>1$ or $\ell-r>1$, so
$\epsilon(c^r_{is}c^{\ell-r}_{sj})=0$ for all $i,j,s$. If $\ell=2$,
then $r>1$, $\ell-r>1$, or $r=1=\ell-r$. The first and second case
clearly imply $\epsilon(c^r_{is}c^{\ell-r}_{sj})=0$. In the final
case, this is also clearly true, unless $i=j$ and strands $i$ and $s$
are paired in the ruling. In this case, either $i<s$ or $s<i=j$, so
either $\epsilon(c^1_{is})=0$ or
$\epsilon(c^1_{sj})=0$. So \[\epsilon\partial
c^\ell_{ii}=\sum_{r=0}^\ell\sum_{s=1}^{N_p}(-1)^{\lvert
  c^r_{is}\rvert+1}\epsilon(c^r_{is}c^{\ell-r}_{si})=0\] for all
$1\leq p\leq k$, $1\leq i\leq N_p$, and $\ell>1$. So for $1<\ell$
\[\epsilon\partial c^\ell_{ij}=0.\]

(grading) From the definition, $a_i$ is augmented only if the
$\rho$-graded normal ruling is switched at $a_i$ and thus
$\rho\big\vert\lvert a_i\rvert$. Since $\lvert
a_i\rvert=\lvert\ta_i\rvert$, the augmentation is $\rho$-graded.

\begin{prop}
  If $\Lambda\subset\#^k(S^1\times S^2)$ is an $n$-component link,
  $\rho\vert2r(\Lambda)$ is even, and $\Lambda$ has a $\rho$-graded
  normal ruling, then the $\rho$-graded augmentation
  $\epsilon:\A(\Lambda)\to F$ constructed above sends $t_1\cdots t_s$
  to $(-1)^n.$
\end{prop}

Thus, if $\Lambda$ is a knot, $\epsilon(t)=-1$ for all even-graded
augmentations $\epsilon$.

\begin{proof}
  Given a $\rho$-graded ruling of $\Lambda$ in $\#^k(S^1\times S^2)$,
  there is a unique way to extend it to a ruling of $S(\Lambda)$ by
  switching at $d_{ji},e_{ij},f_{ji},g_{ij},h_{ji},q_{ij}$ if and only
  if strands $i<j$ are paired in the ruling of $\Lambda$. Let
  $\tepsilon:\A(\Lambda)\to F$ be the $\rho$-graded augmentation
  resulting from the $\rho$-graded normal ruling and let
  $\epsilon:\A(S(\Lambda))\to F$ be the $\rho$-graded augmentation
  resulting from the $\rho$-graded normal ruling of $S(\Lambda)$ as
  constructed in \cite{Leverson} in $\reals^3$. Note that
  \[\frac{\epsilon(t_1\cdots t_s)}{\tepsilon(t_1\cdots
    t_s)}=\left(\prod_{1\leq p\leq
      k}(-1)^{3N_p}\right)\prod_{i,j\text{ paired}}(-1)^6.\] If
  strands $i<j$ are paired near $x=0$ in the ruling of $\Lambda$, then
  the ruling of $S(\Lambda)$ must be switched at
  $d_{ji},e_{ij},f_{ji},g_{ij},h_{ji},$ and $q_{ij}$ with
  configuration $+$(a) since the ruling is $\rho$-graded and $\rho$ is
  even. So there is one additional base point augmented to $-1$ per
  crossing. Thus, there are six additional base points augmented to
  $-1$ for each pair of strands. Each right cusp contributes one extra
  base point augmented to $-1$ and there are three additional right
  cusps for each strand. However, $N_p$ is even for all $1\leq p\leq
  k$ by Corollary \ref{cor:oddNumStrands} and $\epsilon(t_1\cdots
  t_s)=(-1)^n$ by Theorem 1.1 in \cite{Leverson} so we see that
  \[\frac{(-1)^n}{\tepsilon(t_1\cdots t_s)}=1\]
  and so $\tepsilon(t_1\cdots t_s)=(-1)^n.$

\end{proof}

\section{Correspondence for links in $J^1(S^1)$}\label{sec:solidTorus}
Recall that the $1$-jet space of the circle, $J^1(S^1)$, is
diffeomorphic to the solid torus $S^1_x\times\reals^2_{y,z}$ with
contact structure given by $\xi=\ker(dz-ydx)$. As in
\cite{NgSolidTorus}, by viewing $S^1$ as a quotient of the unit
interval, $S^1=[0,1]/(0\sim1)$, we can see Legendrian links in
$J^1(S^1)$ as quotients of arcs in $I\times\reals^2$ with boundary
conditions which are everywhere tangent to the contact planes. Given a
Legendrian link $\Lambda\subset J^1(S^1)$ we will use the methods of
Lavrov-Rutherford in \cite{LavrovSolidTorus} to show the following,
restated from the introduction:

\begin{thmLinksInSolidTorus}
  Let $\Lambda$ be a Legendrian link in $J^1(S^1)$. Given a field $F$,
  the Chekanov-Eliashberg DGA $(\A,\partial)$ over
  $\integers[t_1^{\pm1},\ldots,t_s^{\pm1}]$ has a $\rho$-graded
  augmentation $\epsilon:\A\to F$ if and only if a front diagram of
  $\Lambda$ has a $\rho$-graded generalized normal ruling.
\end{thmLinksInSolidTorus}

We recall the definition of generalized normal ruling as given in
\cite{LavrovSolidTorus}.

\begin{defn}\label{defn:genRuling}
  A {\bf generalized normal ruling} is a sequence of involutions
  $\sigma=(\sigma_1,\ldots,\sigma_M)$ as in Definition
  \ref{defn:normalRuling} with the following differences:
  \begin{enumerate}
  \item Remove the requirement that $\sigma_m$ is fixed-point-free and
    the condition about $1$-handles.
  \item If strands $\ell$ and $\ell+1$ cross in the interval
    $(x_{m-1},x_m)$ above $I_{m-1}$, where exactly one of the crossing
    strands is a fixed point of $\sigma_m$, then the crossing is a
    switch if $\sigma_m$ satisfies the conditions in
    \eqref{cond:switch} of Definition \ref{defn:normalRuling}. If
    crossing is a switch, then we require an additional normality
    condition:
    \[\sigma_m(\ell)=\ell<\ell+1<\sigma_m(\ell+1)\text{ or
    }\sigma_m(\ell)<\ell<\ell+1=\sigma_m(\ell+1).\]
  \end{enumerate}

  A {\bf strictly generalized normal ruling} is a generalized normal
  ruling which is not a normal ruling, in other words, a generalized
  normal ruling with at least one fixed point.
\end{defn}

Thus, near a crossing, a generalized normal ruling looks like the
crossings in Figure \ref{fig:config} or Figure
\ref{fig:configGeneralized}.

\begin{figure}
  \labellist
  \small\hair 2pt
  \pinlabel $(g)$ [t] at 89 10
  \pinlabel $(h)$ [t] at 317 10
  \endlabellist
  \includegraphics[width=.4\textwidth]{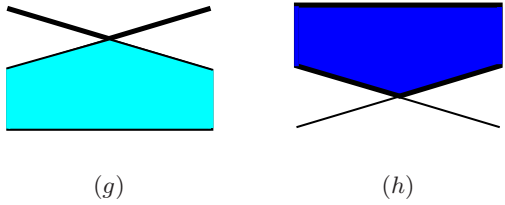}
  \vspace{.25in}
  \caption{These figures give the configuration of a generalized
    normal ruling near a switched crossing involving exactly one
    self-paired strand. With the top row of configurations in Figure
    \ref{fig:config}, these are all possible configurations of a
    generalized normal ruling near a switched crossing.}
  \label{fig:configGeneralized}
\end{figure}

\begin{rmk}
  \begin{enumerate}
  \item If a crossing involving strands $\ell$ and $\ell+1$ occurs in
    the interval $(x_{m-1},x_m)$ and both crossing strands are fixed
    by the ruling, self-paired, in other words,
    $\sigma_{m-1}(\ell)=\ell$ and $\sigma_{m-1}(\ell+1)=\ell+1$, then
    $\sigma_m=(\ell\quad\ell+1)\circ\sigma_{m-1}\circ(\ell\quad\ell+1)$
    and so we will not consider such crossings to be switched.
  \item Note that the number of generalized normal rulings of a
    Legendrian link is not invariant under Legendrian isotopy.
  \end{enumerate}
\end{rmk}

The definition of the Chekanov-Eliashberg DGA of a Legendrian link in
$\reals^3$ can be extended to Legendrian links in $J^1(S^1)$. (One can
find the full definition of the Chekanov-Eliashberg DGA of a
Legendrian link in $J^1(S^1)$ in \cite{NgSolidTorus}.) Note that given
an augmentation of the Chekanov-Eliashberg DGA over
$\integers[t,t^{-1}]$ of a Legendrian link in $S^1\times S^2$, one can
define an augmentation of the DGA of the analogous link (where if a
strand goes through the $1$-handle with $y=y_0$ at $x=0$, then it is
paired with the strand going through the $1$-handle with $y=y_0$ at
$x=A$) in $J^1(S^1)$ and similarly for normal rulings. (The resulting
normal ruling of the link in $J^1(S^1)$ will not have any self-paired
strands.) However, there is no reason to think the converse is true.

\subsection{Matrix definition of the DGA in $J^1(S^1)$}
Ng and Traynor define a version of the Chekanov-Eliashberg DGA $\A$
over $R=\integers[t,t^{-1}]$ in \cite{NgSolidTorus}. For ease of
definition, note that we can assume all left and right cusps involve
the two strands with lowest $z$-coordinate (and thus highest labels)
and that there is one base point at $x=0$ on each strand and these are
the only base points. We give the definition of the DGA for the dipped
version $\Lambda$, $D(\Lambda)$ as in \cite{LavrovSolidTorus}. Label
the dips as in Figure \ref{fig:dips} with $b^m_{ij}$ and $c^m_{ij}$ in
the dip at $x_m$. Place these generators in upper triangular matrices
\[B_m=(b^m_{ij})\text{ and }C_m=(c^m_{ij}).\] Note that since the
$x$-coordinate is $S^1$-valued, we need to add the convention that
$B_0=B_M$ and $C_0=C_M$. We then see that
\begin{align*}
  \partial C_m&=(\Sigma C_m)^2,\\
  \partial B_m&=-\Sigma(I+B_m)\Sigma C_m+\widetilde{C}_{m-1}(I+B_m),
\end{align*}
where $\Sigma$ is the diagonal matrix with $(-1)^{\mu_m(i)}$ the $i$-th
entry on the diagonal for Maslov potential $\mu_m$ at $x=x_m$ and $I$ is the
appropriately sized identity matrix. The form of $\widetilde{C}_m$ will
depend on the tangle appearing in the interval $(x_{m-1},x_m)$.

If $(x_{m-1},x_m)$ contains a crossing $a_m$ of strands $k$ and $k+1$,
then
\begin{align*}
  &\partial a_m=c^{m-1}_{k,k+1}\\
  &\widetilde{C}_{m-1}= U_{k,k+1}\widehat{C}_{m-1}V_{k,k+1},
\end{align*}
where $U_{k,k+1}$ and $V_{k,k+1}$ are the identity matrix with the
$2\times2$ block in rows $k$ and $k+1$ and columns $k$ and $k+1$
replaced with $\begin{pmatrix}0&1\\1&(-1)^{\lvert
    a_m\rvert+1}a_m\end{pmatrix}$ for $U_{k,k+1}$ and $\begin{pmatrix}
  a_m&1\\1&0\end{pmatrix}$ for $V_{k,k+1}$, and $\widehat{C}_{m-1}$ is
$C_{m-1}$ with $0$ replacing the entry $c^{m-1}_{k,k+1}$.

If $(x_{m-1},x_m)$ contains a left cusp, by assumption strands
$N(m)-1$ and $N(m)$ are incident to the cusp. In this case,
\[\widetilde{C}_{m-1}=JC_{m-1}J^T+W,\]
where $J$ is the $N(m-1)\times N(m-1)$ identity matrix with two rows of
zeroes added to the bottom and $W$ is $N(m)\times N(m)$ matrix where
the $(N(m)-1,N(m))$-entry is $1$ and all other entries are zero.

Finally, if $(x_{m-1},x_m)$ contains a right cusp $a_m$, by assumption
strands $N(m)-1$ and $N(m)$ are incident to the cusp. In this case
\begin{align*}
  &\partial a_m=1+c^{m-1}_{N(m-1)-1,N(M-1))},\\
  &\widetilde{C}_{m-1}=KC_{m-1}K^T,
\end{align*}
where $K$ is the $N(m-1)\times N(m-1)$ identity matrix with two
columns of zeroes added to the right.

\subsection{Proof of correspondence}

We will use the methods of \cite{LavrovSolidTorus} to prove
Theorem~\ref{thm:main}. A few conventions and notation: Assume all
left and right cusps occur at lowest $z$-coordinate of all strands at
that $x$-coordinate, in other words, assume for all cusps that the two
strands with highest label are incident to the cusp. Assume that there
is one base point at $x=0$ of $\Lambda$ on each strand and these are
the only base points. Given an involution $\sigma$ of
$\{1,\ldots,N\}$, $\sigma^2=id$, we define $A_\sigma=(a_{ij})$ the
$N\times N$ matrix with entries
\[a_{ij}=\begin{cases}
  1&\text{if }i<\sigma(i)=j\\
  0&\other
\end{cases}\]

(Ruling to augmentation) Given a generalized normal ruling
$\sigma=(\sigma_1,\ldots,\sigma_M)$, we will define a $\rho$-graded
augmentation $\epsilon:\A(D(\Lambda))\to F$ satisfying Property (R)
(as in \cite{SabloffAug}) by defining $\epsilon$ on the crossings in
the dip involving crossings $b^0_{ij}$ and $c^0_{ij}$ and extending to
the right.

\bigno {\bf Property (R):} In any dip, the generator $c^m_{rs}$ is
augmented (to $1$) if and only if $\sigma_m(r)=s$.\bigskip

Add a base point to the loop in each resolution of a right
cusp. Augment all base points to $-1$.  Given a crossing $a$, set
\[\epsilon(a)=\begin{cases}
  1&\text{if the ruling is switched at }a\\
  0&\other.
\end{cases}\] Define $\epsilon(B_0)=0$ and
$\epsilon(C_0)=A_{\sigma_0}$. We will now extend $\epsilon$ to the
right. Suppose $\epsilon$ is defined on all crossings in the interval
$(0,x_{m-1})$. If $(x_{m-1},x_m)$ contains a crossing, define
$\epsilon$ on crossings $b^m_{ij}$ and $c^m_{ij}$ and add base points
as in Figure \ref{fig:basepoints} and Figure
\ref{fig:basepointsGeneralized}. If $(x_{m-1},x_m)$ contains a left
cusp, set
\[\epsilon(B_m)=J\epsilon(B_{m-1})J^T+W.\] If $(x_{m-1},x_m)$ contains
a right cusp, set
\[\epsilon(B_m)=K\epsilon(B_{m-1})K^T.\]
It is easy to check that by our definition the augmentation satisfies
Property (R), which tells us $\epsilon(B_0)=\epsilon(B_M)$ and
$\epsilon(C_0)=\epsilon(C_M)$, and our augmentation is a $\rho$-graded
augmentation.

\begin{figure}
  \labellist
  \small\hair 3pt 
  \pinlabel {(g)} [b] at 217 157 
  \pinlabel $c$ [tl] at 158 76
  \pinlabel $a$ [b] at 217 126 
  \pinlabel $a^{-1}$ [tr] at 306 45
  \pinlabel $ac$ [tl] at 387 77
  \pinlabel $1$ [r] at 275 92
  \pinlabel $2$ [l] at 306 104
  
  \pinlabel {(h)} [b] at 699 157
  \pinlabel $c$ [tl] at 607 44
  \pinlabel $a$ [b] at 699 94
  \pinlabel $a^{-1}$ [tr] at 758 77
  \pinlabel $ac$ [tl] at 834 44
  \endlabellist
  \includegraphics[width=.8\textwidth]{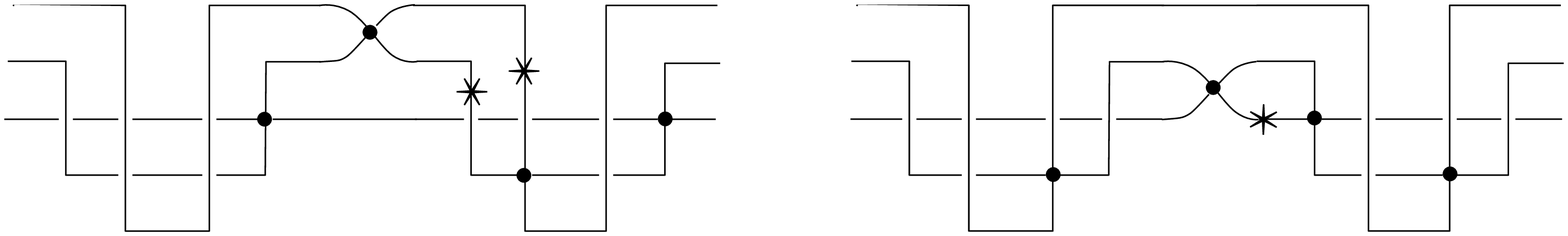}
  \caption{In the diagrams, $*_i$ denotes the base point associated to
    $t_i$. A dot denotes the specified crossing is augmented and the
    augmentation sends the crossing to the label. In configuration
    (g), $\epsilon(t_1)=(-1)^{\lvert a\rvert+1}$ and
    $\epsilon(t_2)=(-1)^{\lvert c_{i,i+1}\rvert+1}$. In
    configuration (h), $\epsilon(t)=-1$.}
  \label{fig:basepointsGeneralized}
\end{figure}

(Augmentation to ruling) This direction of the proof follows that of
the $\integers/2$ case in \cite{LavrovSolidTorus} and is based on
canonical form results from linear algebra due to Barannikov
\cite{Barannikov}.

\begin{defn}
  An {\bf $M$-complex} $(V,\B,d)$ is a vector space $V$ over a field
  $F$ with an ordered basis $\B=\{v_1,\ldots,v_N\}$ and a differential
  $d:V\to V$ of the form $dv_i=\sum_{j=i+1}^Na_{ij}v_j$ satisfying
  $d^2=0$.
\end{defn}

The following two propositions are essentially Proposition 5.4 and 5.6
in \cite{LavrovSolidTorus} and Lemma 2 and 4 in \cite{Barannikov}.

\begin{prop}\label{prop:involution}
  If $(V,\B,d)$ is an $M$-complex, then there exists a triangular
  change of basis $\{\tv_1,\ldots,\tv_N\}$ with
  $\tv_i=\sum_{j=1}^Na_{ij}v_j$ and an involution
  $\tau:\{1,\ldots,N\}\to\{1,\ldots,N\}$ such that
  \[d\tv_i=\begin{cases}
    \tv_j&\text{if }i<\tau(i)=j,\\
    0&\other.
  \end{cases}\] Moreover, the involution $\tau$ is unique.
\end{prop}

\begin{rmk} \label{rmk:involutionProp}
  \begin{enumerate}
  \item If the basis elements $v_i$ have been assigned degrees
    $\lvert v_i\rvert\in\integers/\rho$ such that $V$ is
    $\integers/\rho$-graded and $d$ has degree $-1$, then it can be
    assumed that the change of basis preserves degree. Thus, if
    $i<\tau(i)=j$, then $\lvert v_i\rvert=\lvert v_j\rvert+1$.
  \item The set $\{[\tv_i]:\tau(i)=i\}$ forms a basis for the homology
    $H(V,d)$.
  \item In matrix formulation, Proposition \ref{prop:involution} says
    there is a unique function $D\mapsto\tau(D)$ which assigns an
    involution $\tau=\tau(D)$ to each strictly upper triangular matrix
    $D$ with $D^2=0$ and there is an invertible upper triangular
    matrix $P$ so that $PDP^{-1}=A_\tau$. The uniqueness statement
    tells us that $\tau(QDQ^{-1})=\tau(D)$ if $Q$ is a nonsingular
    upper triangular matrix.
  \end{enumerate}
\end{rmk}

\begin{prop}\label{prop:mcomplexRelation}
  Suppose $(V,\B,d)$ is an $M$-complex and $k\in\{1,\ldots,N\}$ such
  that $dv_k=\sum_{j=k+2}^Na_{kj}v_j$ so the triple $(V,\B',d)$ with
  $\B'=\{v_1,\ldots,v_{k+1},v_k,\ldots,v_N\}$ is also an
  $M$-complex. Then the associated involutions $\tau$ and $\tau'$ from
  Proposition \ref{prop:involution} are related as follows:
  \begin{enumerate}
  \item If
    \begin{align*}
      &\tau(k+1)<\tau(k)<k<k+1,\\
      &\tau(k)<k<k+1<\tau(k+1),\\
      &k<k+1<\tau(k+1)<\tau(k),\\
      &\tau(k)<k<k+1=\tau(k+1),\\
      &\tau(k)=k<k+1<\tau(k+1)\\
    \end{align*}
    then either $\tau'=\tau$ or $\tau'=(k\quad
    k+1)\circ\tau\circ(k\quad k+1)$.
  \item Otherwise $\tau'=(k\quad k+1)\circ\tau\circ(k\quad k+1)$.
  \end{enumerate}
\end{prop}

(Augmentation to ruling) This part of the proof is the same as the
analogous statement in \cite{LavrovSolidTorus} with
$\Sigma\epsilon(C_{m-1})$ replacing $\epsilon(Y_{m-1})$.

\subsection{Corollaries}
The following proposition uses techniques in the proof of Theorem
\ref{thm:correspSolidTorus} to show that
\[\aug_\rho(\Lambda)=F\backslash0\] for any field $F$ and any $\rho$
if $\Lambda$ has a strictly generalized normal ruling.

\begin{prop}
  Given a field $F$ and a Legendrian link $\Lambda\subset J^1(S^1)$
  with $n$ components and a strictly generalized normal ruling, for
  all $0\neq x\in F$ there exists an augmentation $\epsilon:\A\to F$
  such that
  \[\epsilon(t_1\cdots t_s)=x.\]
\end{prop}

\begin{proof}
  Fix $0\neq x\in F$. Given a generalized normal ruling
  $\sigma=(\sigma_1,\ldots,\sigma_M)$ for $\Lambda$ with a self-paired
  strand, we will construct an augmentation
  $\epsilon:\A(D(\Lambda))\to F$ such that $\epsilon(t_1\cdots
  t_s)=x$.

  Suppose $k$ is the label at $x=0$ of a self-paired strand of the
  generalized normal ruling $\sigma$, in other words,
  $\sigma_0(k)=k$. We can assume that $D(\Lambda)$ has one base point
  corresponding to $t_i$ on strand $i$ at $x=0$ and one base point in
  the loop in the resolution of each right cusp, and no other base
  points. Define
  \[\epsilon(t_i)=\begin{cases}
    (-1)^{N+c-1}x&\text{if }i=k,\\
    -1&\other,
  \end{cases}\] where $c$ is the number of right cusps and $N$ is the
  number of strands at $x=0$.

  Define $\epsilon$ on all crossings as in the proof of ruling to
  augmentation in Theorem \ref{thm:correspSolidTorus}. Note that $t_k$
  does not appear on the boundary of any totally augmented disks and
  so $\epsilon$ is still an augmentation, but now
  \[\epsilon(t_1\cdots t_s)=x\]
  as desired.
\end{proof}

\begin{rmk}
  For any link $\Lambda\subset J^1(S^1)$, one can consider the
  analogous link $\Lambda'\subset S^1\times S^2$. Note that
  $\A(\Lambda)\to\A(\Lambda')$ where the map is inclusion. Thus, any
  augmentation $\epsilon':\Lambda'\to F$ gives an augmentation
  $\epsilon:\Lambda\to F$. As one would expect from
  Theorem~\ref{thm:main} and Theorem~\ref{thm:correspSolidTorus}, it
  is also clear that any normal ruling of $\Lambda'\subset S^1\times
  S^2$ gives a generalized normal ruling of $\Lambda\subset J^1(S^1)$.
\end{rmk}

\section*{Appendix}\label{app:nonvanishing}

The appendix will address Corollary~\ref{cor:nonvanishing} which
follows from
\begin{enumerate}
\item Theorem~\ref{thm:main} over $\rationals$ and
\item the result that if a graded augmentation to the rationals exists
  then the full symplectic homology is
  nonzero. \label{en:nonvanishing}
\end{enumerate}
The second result is known to experts. We will outline the proof here
for completeness. Statement~\ref{en:nonvanishing} is a straight
forward consequence of work of Bourgeois, Ekholm, and Eliashberg
\cite{BEE} and has previously been observed in \cite{LidmanSivek}.

Every connected Weinstein (Stein) $4$-manifold $X$ can be decomposed
into $1$- and $2$-handle attachments to $D^4$ along $\partial
D^4=S^3$. Thus, for each such $4$-manifold there exists a Legendrian
link $\Lambda$ in $\#^k(S^1\times S^2)$, the boundary of the
$4$-manifold, so that attaching $2$-handles along $\Lambda$ to
$\#^k(S^1\times S^2)$ results in $X$.

Using the notation of \cite{BEE}, results of Bourgeois, Ekholm, and
Eliashberg in \cite{BEE} tell us that:

\begin{prop}[\cite{BEE} Corollary 5.7]\label{prop:BEE}
  \[\hsh=\hlhho,\]
  where $\hlhho$ is the homology of the Hochschild complex associated
  to the Chekanov-Eliashberg differential graded algebra over
  $\rationals$.
\end{prop}

Therefore, if the DGA for $\Lambda$ has a graded augmentation to
$\rationals$, then $\hsh$ is nonzero. By Theorem~\ref{thm:main}, we
know that the DGA for $\Lambda$ has a graded augmentation to
$\rationals$ if and only if $\Lambda$ has a graded normal
ruling. Thus, restated from the introduction:

\begin{corNonvanishing}
  If $X$ is the Weinstein $4$-manifold that results from attaching
  $2$-handles along a Legendrian link $\Lambda$ to $\#^k(S^1\times
  S^2)$ and $\Lambda$ has a graded normal ruling, then the full
  symplectic homology $\hsh$ is nonzero.
\end{corNonvanishing}

For completeness, we give an outline of the proof of
statement~\ref{en:nonvanishing}. Recall that full symplectic homology
is a symplectic invariant of Weinstein $4$-manifolds which coincides
with the Floer-Hofer symplectic homology.

We will show that given a graded augmentation $\epsilon'$ of the
Chekanov-Eliashberg DGA over $\integers[t,t^{-1}]$ of a Legendrian
knot $\Lambda$ to $\rationals$, one can define a graded augmentation
$\epsilon:\clhho\to\rationals$, where the homology of $\clhho$ is
$\hlhho$. Recall that elements of
$\clhho=\ccheck\oplus\rationals\oplus\chat$ are of the form
$(\check{w},n,\hat{v})$ for some $w,v\in\clho\subset\clha$ and
$n\in\rationals$. Define
\begin{align*}
  \epsilon:\clhho=\ccheck\oplus\rationals\oplus\chat\to\rationals\\
  \epsilon(\check{w},n,\hat{v})=\epsilon'(w)+n
\end{align*}

Let us check that this gives an augmentation. Recall
\[d_{H_0}(\check{w},n,\hat{v})=(\check{d}_{LHO^+}\check{w}+d_{MH_0^+}\hat{v},n,\hat{d}_{LHO^+}\hat{v})=\left(\sum_{j=1}^r\check{w}_j+\check{c}_1c_2\cdots
  c_\ell-c_1\cdots
  c_{\ell-1}\check{c}_\ell,n,\hat{d}_{LHO^+}\hat{v}\right)\] if
$d_{LHO^+}(w)=\sum_{j=1}^rw_j$ and $v=c_1\ldots c_\ell$. Thus, 
\begin{align*}
  \epsilon(d_{H_0}(\check{w},n,\hat{v}))&=\epsilon'\left(\sum_{j=1}^rw_j\right)+\epsilon'(c_1\cdots c_\ell)-\epsilon'(c_1\cdots c_\ell)+n\\
  &=\epsilon'\left(\sum_{j=1}^rw_j+n\right)\\
  &=\epsilon'(d_{LHO}w)=0
\end{align*}
since $\epsilon'$ is an augmentation of $\clha$, $\clho\subset\clha$,
and $d_{LHO}=d_{LHA}\vert_{LHO}$.

One can show that this construction also works if $\epsilon'$ is a pure
augmentation of a link
$\Lambda=\Lambda_1\coprod\cdots\coprod\Lambda_N$, where an
augmentation is {\bf pure} if when a crossing $c$ is augmented, then there
exists $1\leq i\leq N$ such that $c$ is a crossing of $\Lambda_i$.

\bibliographystyle{plain}
\bibliography{S1S2AugmentationsAndRulings}{}

\end{document}